\documentclass{amsart}[11pt]
\usepackage[letterpaper, pdftex, bookmarks=false]{hyperref}
\usepackage{fullpage}
\usepackage{afterpage}
\usepackage{graphicx}
\usepackage{tikz,pgfplots}
\usetikzlibrary{patterns}
\usepackage{amssymb, amsmath, amsthm, mathtools}
\usepackage{algorithmic}
\usepackage{algorithm}
\usepackage[numbers, sort]{natbib}
\bibpunct{[}{]}{;}{n}{,}{,}
\usepackage[all]{xy}
\usepackage{caption}
\usepackage{subcaption}
\usepackage{placeins}

\newcommand{\lefri}[1]{\left(#1\right)}

\newcommand{\dt}[0]{\mbox{d}t}

\newcommand{\dLdq}[0]{\frac{\partial L}{\partial q}}
\newcommand{\dLddq}[0]{\frac{\partial L}{\partial \dot{q}}}

\newcommand{\galqn}[0]{\tilde{q}_{n}}
\newcommand{\dgalqn}[0]{\dot{\tilde{q}}_{n}}
\newcommand{\galpn}[0]{\tilde{p}_{n}}

\newcommand{\optqn}[0]{\hat{q}_{n}}
\newcommand{\doptqn}[0]{\dot{\hat{q}}_{n}}

\newcommand{\truq}[0]{\bar{q}}
\newcommand{\dtruq}[0]{\dot{\bar{q}}}
\newcommand{\trup}[0]{\bar{p}}

\newcommand{\bj}[0]{b_{j}}
\newcommand{\bnj}[0]{b_{n_{j}}}
\newcommand{\cnjh}[0]{c_{n_{j}}h}

\newcommand{\EDL}[2]{L_{d}^{E}\lefri{#1,#2}}
\newcommand{\EDLh}[3]{L_{d}^{E}\lefri{#1,#2,#3}}

\newcommand{\GDL}[2]{L_{d}^{G}\lefri{#1,#2}}
\newcommand{\GDLh}[3]{L_{d}^{G}\lefri{#1,#2,#3}}
\newcommand{\GDLn}[0]{L_{d}^{G}}

\newcommand{\SDLN}[3]{L_{d}^{S}\lefri{#1,#2,#3}}

\newcommand{\CQ}[0]{C^{2}\lefri{\left[0,h\right],Q}}
\newcommand{\FdFSpace}[0]{\mathbb{M}^{n}\lefri{\left[0,h\right],Q}}
\newcommand{\HoQ}[0]{H_{0}^{1}\lefri{\left[0,h\right],Q}}

\newcommand{\SobNorm}[2]{\left\|#1\right\|_{W^{1,#2}\lefri{\left[0,h\right]}}}
\newcommand{\LNorm}[2]{\left\|#1\right\|_{L^{#2}\lefri{\left[0,h\right]}}}
\newcommand{\MNorm}[2]{\left\|#1\right\|_{#2}}

\newcommand{\ApproxC}[0]{C_{A}}
\newcommand{\ApproxK}[0]{K_{A}}
\newcommand{\SpecK}[0]{K_{s}}
\newcommand{\SpecC}[0]{C_{s}}
\newcommand{\QuadC}[0]{C_{g}}
\newcommand{\QuadK}[0]{K_{g}}
\newcommand{\CoerC}[0]{C_{f}}
\newcommand{\OptC}[0]{C_{op}}
\newcommand{\akC}[0]{C_{a}}
\newcommand{\LipC}[0]{C_{L}}
\newcommand{\LagLipC}[0]{L_{\alpha}}
\newcommand{\NoeC}[0]{C_{v}}

\newcommand{\truqargsf}[2]{\substack{q \in \CQ\\q\lefri{0} = #1, q\lefri{h} = #2}}

\newcommand{\galargsf}[2]{\substack{q_{n} \in \FdFSpace\\q_{n}\lefri{0} = #1, q_{n}\lefri{h} = #2}}
\newcommand{\galargsvf}[4]{\substack{q_{n} \in \FdFSpace\\#1 = #2, #3 = #4}}

\newcommand{\truqargs}[0]{\substack{q\lefri{0} = q_{k}\\q\lefri{h} = q_{k+1}\\q \in \CQ}} 

\newcommand{\quaderr}[2]{e_{q}\lefri{#1,#2}}

\newcommand{\ext}[0]{\operatornamewithlimits{ext}}
\newcommand{\argmin}[0]{\operatornamewithlimits{argmin}}

\newtheorem{theorem}{Theorem}[section]
\newtheorem{lemma}{Lemma}[section]

\begin{document}
\title{Spectral Variational Integrators}

\author[J.~Hall]{James Hall}
\address{Department of Mathematics\\
University of California, San Diego\\
9500 Gilman Drive \#0112\\
La Jolla, California 92093-0112, USA}
\email{j9hall@math.ucsd.edu}

\author[M.~Leok]{Melvin Leok}
\address{Department of Mathematics\\
University of California, San Diego\\
9500 Gilman Drive \#0112\\
La Jolla, California 92093-0112, USA}
\email{mleok@math.ucsd.edu}

\begin{abstract}
In this paper, we present a new variational integrator for problems in Lagrangian mechanics. Using techniques from Galerkin variational integrators, we construct a scheme for numerical integration that converges geometrically, and is symplectic and momentum preserving. Furthermore, we prove that under appropriate assumptions, variational integrators constructed using Galerkin techniques will yield numerical methods that are in a certain sense optimal, converging at the same rate as the best possible approximation in a certain function space. We further prove that certain geometric invariants also converge at an optimal rate, and that the error associated with these geometric invariants is independent of the number of steps taken. We close with several numerical examples that demonstrate the predicted rates of convergence.  
\end{abstract}

\maketitle

\allowdisplaybreaks

\section{Introduction}

There has been significant recent interest in the development of structure-preserving numerical methods for variational problems. One of the key points of interest is developing high-order symplectic integrators for Lagrangian systems. The generalized Galerkin framework has proven to be a powerful theoretical and practical tool for developing such methods. This paper presents a high-order Galerkin variational integrator that exhibits geometric convergence to the true flow of a Lagrangian system. In addition, this method is symplectic, momentum-preserving, and stable even for very large time steps.

Galerkin variational integrators fall into the general framework of discrete mechanics. For a general and comprehensive introduction to the subject, the reader is referred to \citet{MaWe2001}. Discrete mechanics develops mechanics from discrete variational principles, and, as Marsden and West demonstrated, gives rise to many discrete structures which are analogous to structures found in classical mechanics. By taking these structures into account, discrete mechanics suggests numerical methods which often exhibit excellent long term stability and qualitative behavior. Because of these qualities, much recent work has been done on developing numerical methods from the discrete mechanics viewpoint. See, for example, \citet{HaLuWa2006} for a broad overview of the field of geometric numerical integration, and \citet{MuOr2004, MaWe2001, PaCu2009} discuss the error analysis of variational integrators. Various extensions have also been considered, including, \citet{LaWe2006, LeZh2009} for Hamiltonian systems; \citet{FeMaOrWe2003} for nonsmooth problems with collisions; \citet{MaPaSh1998, LeMaOrWe2003} for Lagrangian PDEs; \citet{CoMa2001,McPe2006, FeZe2005} for nonholonomic systems; \citet{BoOw2009, BoOw2010} for stochastic Hamiltonian systems; \citet{LeLeMc2007, LeLeMc2009, BoMa2009} for problems on Lie groups and homogeneous spaces.

The fundamental object in discrete mechanics is the discrete Lagrangian \(L_{d}: Q \times Q \times \mathbb{R} \rightarrow \mathbb{R}\), where \(Q\) is a configuration manifold. The discrete Lagrangian is chosen to be an approximation to the action of a Lagrangian over the time step \(\left[0,h\right]\),%
\begin{align*}
L_{d}\lefri{q_{0},q_{1},h} \approx \ext_{\truqargsf{q_{0}}{q_{1}}} \int_{0}^{h} L\lefri{q,\dot{q}} \dt,
\end{align*}%
or simply \(L_{d}\lefri{q_{0},q_{1}}\) when \(h\) is assumed to be constant. Discrete mechanics is formulated by finding stationary points of a discrete action sum based on the sum of discrete Lagrangians,%
\begin{align*}
\mathbb{S}\lefri{\left\{q_{k}\right\}_{k=1}^{n}} = \sum_{k=1}^{n-1} L_{d}\lefri{q_{k},q_{k+1}} \approx \int_{t_{1}}^{t_{2}} L\lefri{q,\dot{q}}\dt.
\end{align*}
For Galerkin variational integrators specifically, the discrete Lagrangian is induced by constructing a discrete approximation of the action integral over the interval \(\left[0,h\right]\) based on a finite-dimensional function space and quadrature rule. Once this discrete action is constructed, the discrete Lagrangian can be recovered by solving for stationary points of the discrete action subject to fixed endpoints, and then evaluating the discrete action at these stationary points,%
\begin{align}
L_{d}\lefri{q_{0},q_{1},h} = \ext_{\galargsf{q_{0}}{q_{1}}} h\sum_{j=1}^{m} b_{j}L\lefri{q\lefri{c_{j}h},\dot{q}\lefri{c_{j}h}}. \label{discaction}
\end{align}%
Because the rate of convergence of the approximate flow to the true flow is related to how well the discrete Lagrangian approximates the true action, this type of construction gives a method for constructing and analyzing high-order methods. The hope is that the discrete Lagrangian inherits the accuracy of the function space used to construct it, much in the same way as standard finite-element methods. We will show that for certain Lagrangians, Galerkin constructions based on high-order approximation spaces do in fact result in correspondingly high order methods.

Significant work has already been done constructing and analyzing these types of Galerkin variational integrators. In \citet{Le04}, a number of different possible constructions based on the Galerkin framework are presented. In \citet{LeSh2011}, Hermite polynomials are used to construct globally smooth high-order methods. What separates this work from the work that precedes it is the use of a spectral approximation paradigm, which induces methods that exhibit geometric convergence. This type of convergence is established theoretically and demonstrated through numerical examples.

\subsection{Discrete Mechanics}

Before discussing the construction and convergence of spectral variational integrators, it is useful to review some of the fundamental results from discrete mechanics that are used in our analysis. We have already introduced the \emph{discrete Lagrangian} \(L_{d}:Q \times Q \times \mathbb{R} \rightarrow \mathbb{R}\),%
\begin{align*}
L_{d} \lefri{q_{0},q_{1},h} \approx \ext_{\truqargsf{q_{0}}{q_{1}}} \int_{0}^{h}L\lefri{q,\dot{q}}\dt.
\end{align*}%
and the \emph{discrete action sum},%
\begin{align*}
\mathbb{S}\lefri{\left\{q_{k}\right\}_{k=1}^{n}} = \sum_{k=1}^{n-1} L_{d}\lefri{q_{k},q_{k+1}} \approx \int_{t_{1}}^{t_{2}} L\lefri{q,\dot{q}}\dt.
\end{align*}%
Taking variations of the discrete action sum and using discrete integration by parts leads to the discrete Euler-Lagrange equations,%
\begin{align}
D_{2}L_{d}\lefri{q_{k-1},q_{k}} + D_{1}L_{d}\lefri{q_{k},q_{k+1}} = 0, \label{DEL}
\end{align}
where \(D_{1}\) denotes differentiation with respect to the first argument and \(D_{2}\) denotes differentiation with respect to the second argument. Given \(\lefri{q_{k-1},q_{k}}\), these equations implicitly define an update map, known as the \textit{discrete Lagrangian flow map}, \(F_{L_{d}}: Q \times Q \rightarrow Q \times Q\), given by \(F_{L_{d}}\lefri{q_{k-1},q_{k}} = \lefri{q_{k},q_{k+1}}\), where \(\lefri{q_{k-1},q_{k}},\lefri{q_{k},q_{k+1}}\) satisfy (\ref{DEL}). Furthermore, the discrete Lagrangian defines the \textit{discrete Legendre transforms}, \(\mathbb{F}^{\pm}L_{d}: Q\times Q \rightarrow T^{*}Q\):%
\begin{align*}
\mathbb{F}^{+}L_{d}&:\lefri{q_{0},q_{1}} \rightarrow \lefri{q_{1},p_{1}} = \lefri{q_{1},D_{2}L_{d}\lefri{q_{0},q_{1}}}, \\
\mathbb{F}^{-}L_{d}&:\lefri{q_{0},q_{1}} \rightarrow \lefri{q_{0},p_{0}} = \lefri{q_{0},-D_{1}L_{d}\lefri{q_{0},q_{1}}}.
\end{align*}%
Using the discrete Legendre transforms, we define the \textit{discrete Hamiltonian flow map}, \(\tilde{F}_{L_{d}}: T^{*}Q \rightarrow T^{*}Q\),%
\begin{align*}
\tilde{F}_{L_{d}} &: \lefri{q_{0},p_{0}} \rightarrow \lefri{q_{1},p_{1}} = \mathbb{F}^{+}L_{d}\lefri{\lefri{\mathbb{F}^{-}L_{d}}^{-1}\lefri{q_{0},p_{0}}}. 
\end{align*}%
The following commutative diagram illustrates the relationship between the discrete Hamiltonian flow map, discrete Lagrangian flow map, and the discrete Legendre transforms,%
\begin{align*}
\xymatrix{
& \lefri{q_{k},p_{k}}  \ar[rr]^{\tilde{F}_{L_{d}}} & & \lefri{q_{k+1},p_{k+1}} & \\
& & & & \\
\lefri{q_{k-1},q_{k}} \ar[uur]^{\mathbb{F}^{+}L_{d}} \ar[rr]_{F_{L_{d}}} & & \lefri{q_{k},q_{k+1}} \ar[rr]_{F_{L_{d}}} \ar[uur]^{\mathbb{F}^{+}L_{d}} \ar[uul]_{\mathbb{F}^{-}L_{d}}& &\lefri{q_{k+1},q_{k+2}} \ar[uul]_{\mathbb{F}^{-}L_{d}}
}
\end{align*}
We now introduce the \textit{exact discrete Lagrangian} \(L^{E}_{d}\),%
\begin{align*}
L^{E}_{d} \lefri{q_{0},q_{1},h} = \ext_{\truqargs{q_{0}}{q_{1}}}\int_{0}^{h}L\lefri{q,\dot{q}}\dt.
\end{align*}%
An important theoretical result for the error analysis of variational integrators is that the discrete Hamiltonian and Lagrangian flow maps associated with the exact discrete Lagrangian produces an exact sampling of the true flow, as was shown in \citet{MaWe2001}. Using this result, \citet{MaWe2001} shows that there is a fundamental relationship between how well a discrete Lagrangian \(L_{d}\) approximates the exact discrete Lagrangian \(L_{d}^{E}\) and how well the corresponding discrete Hamiltonian flow maps, discrete Lagrangian flow maps and discrete Legendre transforms approximate each other. Since the exact discrete Lagrangian produces an exact sampling of the true flow, this in turn leads to the following theorem regarding the error analysis of variational integrators, also found in \citet{MaWe2001}:%
\begin{theorem} \emph{(Variational Error Analysis)} \label{MarsConv}
Given a regular Lagrangian \(L\) and corresponding Hamiltonian \(H\), the following are equivalent for a discrete Lagrangian \(L_{d}\):
\begin{enumerate}
\item the discrete Hamiltonian flow map for \(L_{d}\) has error \(\mathcal{O}\lefri{h^{p+1}}\),
\item the discrete Legendre transforms of \(L_{d}\) have error \(\mathcal{O}\lefri{h^{p+1}}\),
\item \(L_{d}\) approximates the exact discrete Lagrangian with error \(\mathcal{O}\lefri{h^{p+1}}\).
\end{enumerate}%
\end{theorem}%
\noindent We will make extensive use of this theorem later when we analyze the convergence of spectral variational integrators.

In addition, in \citet{MaWe2001}, it is shown that integrators constructed in this way, which are referred to as \textit{variational integrators}, have significant geometric structure. Most importantly, variational integrators always conserve the canonical symplectic form, and a discrete Noether's Theorem guarantees that a discrete momentum map is conserved for any continuous symmetry of the discrete Lagrangian. The preservation of these discrete geometric structures underlie the excellent long term behavior of variational integrators.

\section{Construction}

\subsection{Generalized Galerkin Variational Integrators}

\begin{figure}[htbp]
\[
\begin{xy} 0;<45mm,0cm>:<0cm,10mm>::
(-0.1,4.6)*\txt{$Q$},
(2.65,-0.3)*\txt{$t$},
(-0.1,-0.3)*\txt{$0$},
(0.5,-0.3)*\txt{$d_1 h$},
(1,-0.3)*\txt{$d_2 h$},
(1.5,-0.3)*\txt{$d_{i-2} h$},
(2,-0.3)*\txt{$d_{i-1} h$},
(2.5,-0.3)*\txt{$h$},
(0.25,-0.3)*\txt{\(c_{1}h\)},
(0.75,-0.3)*\txt{\(c_{2}h\)},
(1.75,-0.3)*\txt{\(c_{j-1}h\)},
(2.25,-0.3)*\txt{\(c_{j}h\)},
(0,2)*@{*};
(1,1)*@{*}
**\crv{(0.5,3)}
?(.5)*@{*} * !LD!/^-5pt/{q_k^1},
?(0.25)*\txt{\(\boldmath{\times}\)};
?(0.75)*\txt{\(\boldmath{\times}\)};
(1,1);(1.5,2) **\crv{~*=<4pt>{.} (1.2,0)},
(1.5,2)*@{*};(2.5,4)*@{*}
**\crv{(2,4.3)}
?(.5)*@{*} *!LD!/^-5pt/{q_k^{i-1}},
?(0.25)*\txt{\(\boldmath{\times}\)};
?(0.75)*\txt{\(\boldmath{\times}\)};
(0.1,1.7)* !LD!{q_k^0},
(1,1.3)* !LD!{q_k^2},
(1.45,2.2)* !LD!{q_k^{i-2}},
(2.5,4.2)* !LD!{q_k^i},
\ar@{-} (-0.1,0);(1,0),
\ar@{.} (1,0);(1.5,0),
\ar (1.5,0);(2.6,0),
\ar (0,-0.3);(0,4.5),
\ar@{-} (0.5,-0.1);(0.5,0.1),
\ar@{-} (1.0,-0.1);(1.0,0.1),
\ar@{-} (1.5,-0.1);(1.5,0.1),
\ar@{-} (2.0,-0.1);(2,0.1),
\ar@{-} (2.5,-0.1);(2.5,0.1),
\ar@{-} (0.25,-0.1);(.25,0.1),
\ar@{-} (0.75,-0.1);(.75,0.1),
\ar@{-} (1.75,-0.1);(1.75,0.1),
\ar@{-} (2.25,-0.1);(2.25,0.1),
\end{xy}
\]
\caption{A visual schematic of the curve \(\galqn\lefri{t} \in \FdFSpace\). The points marked with (\(\times\)) represent the quadrature points, which may or may not be the same as interpolation points \(d_{i}h\). In this figure we have chosen to depict a curve constructed from interpolating basis functions, but this is not necessary in general.}
\label{GalCurveVis}
\end{figure}
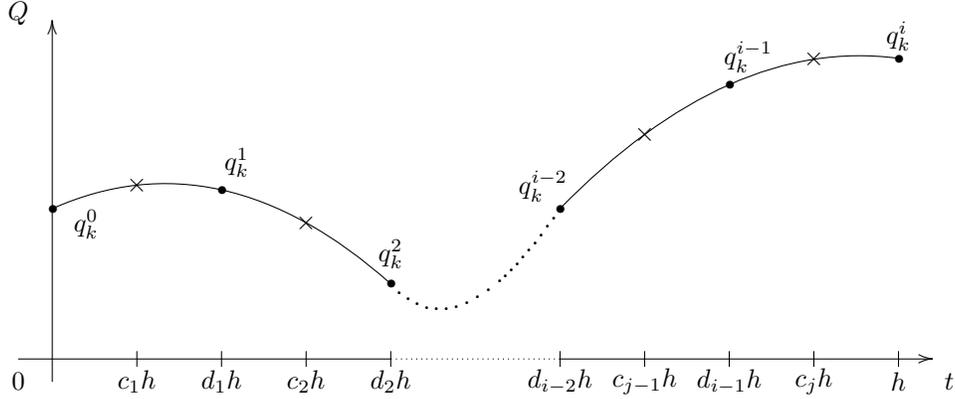

The construction of spectral variational integrators falls within the framework of generalized Galerkin variational integrators, discussed in \citet{Le04} and \citet{MaWe2001}. The motivating idea is that we try to replace the generally non-computable exact discrete Lagrangian \(\EDL{q_{k}}{q_{k+1}}\) with a computable discrete analogue, \(\GDL{q_{k}}{q_{k+1}}\). Galerkin variational integrators are constructed by using a finite-dimensional function space to discretize the action of a Lagrangian. Specifically, given a Lagrangian \(L:TQ \rightarrow \mathbb{R}\), to construct a Galerkin variational integrator:%
\begin{enumerate}
\item choose an \(n\)-dimensional function space \(\FdFSpace \subset \CQ\), with a finite set of basis functions \(\left\{\phi_{i}\lefri{t}\right\}_{i=1}^{n}\), 
\item choose a quadrature rule \(\mathcal{G}\lefri{\cdot}:F\lefri{\left[0,h\right],\mathbb{R}}\rightarrow\mathbb{R}\), so that \(\mathcal{G}\lefri{f} = h\sum_{j=1}^{m} b_{j}f\lefri{c_{j}h} \approx \int_{0}^{h} f\lefri{t} \dt\), where \(F\) is some appropriate function space,
\end{enumerate}%
and then construct the discrete action \(\mathbb{S}_{d}\lefri{\left\{q^{i}_{k}\right\}_{i=1}^{n}}:\prod_{i=1}^{n} Q_{i} \rightarrow \mathbb{R}\), (not to be confused with the discrete action sum \(\mathbb{S}\lefri{\left\{q_{k}\right\}_{k=1}^{\infty}}\)),%
\begin{align*}
\mathbb{S}_{d}\lefri{\left\{q_{k}^{i}\right\}_{i=1}^{n}} = \mathcal{G}\lefri{L\lefri{\sum_{i=1}^{n} q_{k}^{i}\phi_{i}\lefri{t}, \sum_{i=1}^{n}q_{k}^{i}\dot{\phi}_{i}\lefri{t}}} =  h\sum_{j=1}^{m} b_{j} L\lefri{\sum_{i=1}^{n} q_{k}^{i} \phi_{i}\lefri{c_{j}h}, \sum_{i=1}^{n} q_{k}^{i} \dot{\phi}_{i}\lefri{c_{j}h}},
\end{align*}%
where we use superscripts to index the weights associated with each basis function, as in \citet{MaWe2001}.

Once the discrete action has been constructed, a discrete Lagrangian can be induced by finding stationary points \(\tilde{q}_{n}\lefri{t} = \sum_{i=1}^{n}q_{k}^{i}\phi_{i}\lefri{t}\) of the action under the conditions \(\tilde{q}_{n}\lefri{0} = \sum_{i=1}^{n} q_{k}^{i} \phi_{i}\lefri{0} = q_{k}\) and \(\tilde{q}_{n}\lefri{h} = \sum_{i=1}^{n} q_{k}^{i} \phi_{i}\lefri{h} = q_{k+1}\) for some given \(q_{k}\) and \(q_{k+1}\),%
\begin{align*}
L_{d}\lefri{q_{k},q_{k+1},h} = \ext_{\substack{\tilde{q}_{n}\lefri{0} = q_{k}\\ \tilde{q}_{n}\lefri{h} = q_{k+1}}} \mathbb{S}_{d}\lefri{\left\{q_{k}^{i}\right\}_{i=1}^{n}} = \ext_{\substack{\tilde{q}_{n}\lefri{0} = q_{k}\\ \tilde{q}_{n}\lefri{h} = q_{k+1}}} h\sum_{j=1}^{m} b_{j}L\lefri{\tilde{q}_{n}\lefri{c_{j}h},\dot{\tilde{q}}_{n}\lefri{c_{j}h}}.
\end{align*}%
A discrete Lagrangian flow map that result from this type of discrete Lagrangian is referred to as a Galerkin variational integrator.

\subsection{Spectral Variational Integrators}%
There are two defining features of spectral variational integrators. The first is the choice of function space \(\FdFSpace\), and the second is that convergence is achieved not by shortening the time step \(h\), but by increasing the dimension \(n\) of the function space.

\subsubsection{Choice of Function Space}

Restricting our attention to the case where \(Q\) is a linear space, spectral variational integrators are constructed using the basis functions \(\phi_{i}\lefri{t} = l_{i}\lefri{t}\), where \(l_{i}\lefri{t}\) are Lagrange interpolating polynomials based on the points \(h_{i} = \frac{h}{2}\cos\lefri{\frac{i \pi}{n}} + \frac{h}{2}\) which are the Chebyshev points \(t_{i} = \cos\lefri{\frac{i \pi}{n}}\), rescaled and shifted from \(\left[-1,1\right]\) to \(\left[0,h\right]\). The resulting finite dimensional function space \(\FdFSpace\) is simply the polynomials of degree at most \(n\) on \(Q\). However, the choice of this particular set of basis functions offer several advantages over other possible bases for the polynomials:%
\begin{enumerate}
\item the restriction on variations \(\sum_{i=1}^{n} \delta q_{k}^{i} \phi_{i}\lefri{0} = \sum_{i=1}^{n} \delta q_{k}^{i} \phi_{i}\lefri{h} = 0\) reduces to \(\delta q_{k}^{1} = \delta q_{k}^{n} = 0\),
\item the condition \(\galqn\lefri{0} = q_{k}\) reduces to \(q_{k}^{1} = q_{k}\), 
\item the induced numerical methods have generally better stability properties because of the excellent approximation properties of the interpolation polynomials at the Chebyshev points.
\end{enumerate}%
Using this choice of basis functions, for any chosen quadrature rule, the discrete Lagrangian becomes,%
\begin{align*}
L_{d}\lefri{q_{k},q_{k+1},h} = \ext_{\galargsvf{q_{k}^{1}}{q_{k}}{q_{k}^{n}}{q_{k+1}}} h\sum_{j=1}^{m} b_{j} L\lefri{\galqn\lefri{c_{j}h},\dgalqn\lefri{c_{j}h}}.
\end{align*}%
Requiring the curve \(\galqn\lefri{t}\) to be a stationary point of the discretized action provides \(n-2\) internal stage conditions:%
\begin{align}
h\sum_{j=1}^{m} b_{j}\lefri{\dLdq \lefri{\galqn\lefri{c_{j}h},\dgalqn\lefri{c_{j}h}}\phi_{p}\lefri{c_{j}h} + \dLddq \lefri{\galqn \lefri{c_{j}h}, \dgalqn\lefri{c_{j}h}}\dot{\phi}_{p}\lefri{c_{j}h}} &= 0, & p = 2,...,n-1. \label{InterStage}
\end{align}
Combining these internal stage conditions with the discrete Euler-Lagrange equations,%
\begin{align*}
D_{1}L_{d}\lefri{q_{k-1},q_{k}} + D_{2}L_{d}\lefri{q_{k},q_{k+1}} = 0,
\end{align*}%
and the continuity condition \(q_{k}^{1} = q_{k}\) yields the following set of \(n\) nonlinear equations,%
\begin{align}
q_{k}^{1} &= q_{k}, & \nonumber \\
h\sum_{j=1}^{m} b_{j}\lefri{\dLdq \lefri{\galqn\lefri{c_{j}h},\dgalqn\lefri{c_{j}h}}\phi_{p}\lefri{c_{j}h} + \dLddq \lefri{\galqn \lefri{c_{j}h}, \dgalqn\lefri{c_{j}h}}\dot{\phi}_{p}\lefri{c_{j}h}} &= 0, & p = 2,...,n-1, \nonumber\\
h\sum_{j=1}^{m} b_{j}\lefri{\dLdq \lefri{\galqn\lefri{c_{j}h},\dgalqn\lefri{c_{j}h}}\phi_{1}\lefri{c_{j}h} + \dLddq \lefri{\galqn \lefri{c_{j}h}, \dgalqn\lefri{c_{j}h}}\dot{\phi}_{1}\lefri{c_{j}h}} &= p_{k-1}, & \label{momcon}
\end{align}%
which must be solved at each time step \(k\), and the momentum condition:%
\begin{align*}
h\sum_{j=1}^{m} b_{j}\lefri{\dLdq \lefri{\galqn\lefri{c_{j}h},\dgalqn\lefri{c_{j}h}}\phi_{n}\lefri{c_{j}h} + \dLddq \lefri{\galqn \lefri{c_{j}h}, \dgalqn\lefri{c_{j}h}}\dot{\phi_{n}}\lefri{c_{j}h}} &= p_{k},
\end{align*}%
which defines (\ref{momcon}) for the next time step. Evaluating \(q_{k+1} = \galqn\lefri{h}\) defines the next step for the discrete Lagrangian flow map:%
\begin{align*}
F_{L_{d}}\lefri{q_{k-1},q_{k}} = \lefri{q_{k},q_{k+1}},
\end{align*}%
and because of the choice of basis functions, this is simply \(q_{k+1} = q_{k}^{n}\).

\subsubsection{\(n\)-Refinement}

As is typical for spectral numerical methods (see, for example, \citet{Bo2001, Tr2000}), convergence for spectral variational integrators is achieved by increasing the dimension of the function space, \(\FdFSpace\). Furthermore, because the order of the discrete Lagrangian also depends on the order of the quadrature rule \(\mathcal{G}\), we must also refine the quadrature rule as we refine \(n\). Hence, for examining convergence, we must also consider the quadrature rule as a function of \(n\), \(\mathcal{G}_{n}\). Because of the dependence on \(n\) instead of \(h\), we will often examine the discrete Lagrangian \(L_{d}\) as a function of \(Q \times Q \times \mathbb{N}\), %
\begin{align*}
L_{d}\lefri{q_{k},q_{k+1},n} =  \ext_{\galargsvf{q_{k}^{1}}{q_{k}}{q_{k}^{n}}{q_{k+1}}} \mathcal{G}_{n}\lefri{L\lefri{\galqn\lefri{t},\dgalqn\lefri{t}}} = \ext_{\galargsvf{q_{k}^{1}}{q_{k}}{q_{k}^{n}}{q_{k+1}}} h\sum_{j=1}^{m_{n}} \bnj L\lefri{\galqn\lefri{\cnjh},\dgalqn\lefri{\cnjh}}, 
\end{align*}%
as opposed to the more conventional%
\begin{align*}
L_{d}\lefri{q_{k},q_{k+1},h} = \ext_{\galargsvf{q_{k}^{1}}{q_{k}}{q_{k}^{n}}{q_{k+1}}} \mathcal{G}\lefri{L\lefri{\galqn\lefri{t},\dgalqn\lefri{t}}} = \ext_{\galargsvf{q_{k}^{1}}{q_{k}}{q_{k}^{n}}{q_{k+1}}} h\sum_{j=1}^{m} b_{j} L\lefri{\galqn\lefri{c_{j}h},\dgalqn\lefri{c_{j}h}}. 
\end{align*}%
This type of refinement is the foundation for the exceptional convergence properties of spectral variational integrators.

\section{Existence, Uniqueness and Convergence}

In this section, we will discuss the major important properties of Galerkin variational integrators and spectral variational integrators. The first will be the existence of unique solutions to the internal stage equations (\ref{InterStage}) for certain types of Lagrangians. The second is the convergence of the one-step map that results from the Galerkin and spectral variational constructions, which will be shown to be optimal in a certain sense. The third and final is the convergence of continuous approximations to the Euler-Lagrange flow which can easily be constructed from Galerkin and spectral variational integrators, and the behavior of geometric invariants associated with the approximate continuous flow. We will show a number of different convergence results associated with these quantities, which demonstrate that Galerkin and spectral variational integrators can be used to compute continuous approximations to the exact solutions of the Euler-Lagrange equations which have excellent convergence and geometric behavior.

\subsection{Existence and Uniqueness}

In general, demonstrating that there exists a unique solution to the internal stage equations for a spectral variational integrator is difficult, and depends on the properties of the Lagrangian. However, assuming a Lagrangian of the form%
\begin{align*}
L\lefri{q,\dot{q}} = \frac{1}{2} \dot{q}^{T}M\dot{q} - V\lefri{q},
\end{align*}%
it is possible to show the existence and uniqueness of the solutions to the implicit equations for the one-step method under appropriate assumptions.
%
%
\begin{theorem}\label{ExisUniq}\emph{(Existence and Uniqueness of Solutions to the Internal Stage Equations)} Given a Lagrangian \(L:TQ \rightarrow \mathbb{R}\) of the form%
\begin{align*}
L\lefri{q,\dot{q}} = \frac{1}{2}\dot{q}^{T}M\dot{q} - V\lefri{q},
\end{align*}%
if \(\nabla V\) is Lipschitz continuous, \(\bj > 0\) for every \(j\) and \(\sum_{i=1}^{m} \bj = 1\), and \(M\) is symmetric positive-definite, then there exists an interval \(\left[0,h\right]\) where there exists a unique solution to the internal stage equations for a spectral variational integrator.
\end{theorem}
%
%
\begin{proof} We will consider only the case where \(q\lefri{t}\in \mathbb{R}\), but the argument generalizes easily to higher dimensions. To begin, we note that for a Lagrangian of the form,%
\begin{align*}
L\lefri{q,\dot{q}} = \frac{1}{2} \dot{q}^{T}M\dot{q} - V\lefri{q}
\end{align*}%
the equations%
\begin{align*}
q_{k}^{1} =& q_{k}, & \\
h\sum_{j=1}^{m} b_{j}\lefri{\dLdq \lefri{\galqn\lefri{c_{j}h},\dgalqn\lefri{c_{j}h}}\phi_{p}\lefri{c_{j}h} + \dLddq \lefri{\galqn \lefri{c_{j}h}, \dgalqn\lefri{c_{j}h}}\dot{\phi}_{p}\lefri{c_{j}h}} =& 0, & p = 2,...,n-1, \\
h\sum_{j=1}^{m} b_{j}\lefri{\dLdq \lefri{\galqn\lefri{c_{j}h},\dgalqn\lefri{c_{j}h}}\phi_{1}\lefri{c_{j}h} + \dLddq \lefri{\galqn \lefri{c_{j}h}, \dgalqn\lefri{c_{j}h}}\dot{\phi}_{1}\lefri{c_{j}h}} =& p_{k-1}, & 
\end{align*}%
take the form%
\begin{align}
Aq^{i} - f\lefri{q^{i}} = 0, \label{InternalDEL}
\end{align}%
where \(q^{i}\) is the vector of internal weights, \(q^{i} = \lefri{q_{k}^{1},q_{k}^{2},...,q_{k}^{n}}^{T}\), \(A\) is a matrix with entries defined by 
\begin{align}
A_{1,1} =& 1, & & \label{AMatrix1}\\
A_{1,i} =& 0, & i = 2,...,n, & \label{AMatrix2}\\
A_{p,i} =& h\sum_{j=1}^{m} b_{j}M\dot{\phi}_{i}\lefri{c_{j}h}\dot{\phi}_{p}\lefri{c_{j}h}, & p = 2,...,n; & i = 1,...,n, \label{AMatrix3}
\end{align}%
and \(f\) is a vector valued function defined by%
\begin{align*}
f\lefri{q^{i}} = \left(\begin{array}{c}
q_{k} \\
h\sum_{j=1}^{m} b_{j} \nabla V\lefri{\sum_{i=1}^{n}q_{k}^{i} \phi_{i}\lefri{c_{j}h}}\phi_{2}\\
\vdots \\
h\sum_{j=1}^{m} b_{j} \nabla V\lefri{\sum_{i=1}^{n}q_{k}^{i} \phi_{i}\lefri{c_{j}h}}\phi_{n-1}\\
p_{k-1}
\end{array}\right).
\end{align*}%
It is important to note that the entries of \(A\) depend on \(h\). For now we will assume \(A\) is invertible, and that \(\left\|A^{-1}\right\| < \left\|A_{1}^{-1}\right\|\), for where \(A_{1}\) is the matrix \(A\) generated on the interval \(\left[0,1\right]\). Of course, the properties of \(A\) depend on the choice of basis functions \(\left\{\phi_{i}\right\}_{i=1}^{n}\), but we will establish these properties for the polynomial basis later. Defining the map:%
\begin{align*}
\Phi\lefri{q^{i}} = A^{-1}f\lefri{q^{i}},
\end{align*}%
it is easily seen that (\ref{InternalDEL}) is satisfied if and only if \(q^{i} = \Phi\lefri{q^{i}}\), that is, \(q^{i}\) is a fixed point of \(\Phi\lefri{\cdot}\). If we establish that \(\Phi\lefri{\cdot}\) is a contraction mapping,%
\begin{align*}
\left\|\Phi\lefri{w^{i}} - \Phi\lefri{v^{i}}\right\|_{\infty} \leq k \left\|w^{i} - v^{i}\right\|_{\infty},
\end{align*}%
for some \(k < 1\), we can establish the existence of a unique fixed point, and thus show that the steps of the one step method are well-defined. Here, and throughout this section, we use \(\MNorm{\cdot}{p}\) to denote the vector or matrix \(p\)-norm, as appropriate.

To show that \(\Phi\lefri{\cdot}\) is a contraction mapping, we consider arbitrary \(w^{i}\) and \(v^{i}\):%
\begin{align*}
\left\|\Phi\lefri{w^{i}} - \Phi\lefri{v^{i}}\right\|_{\infty} &= \left\|A^{-1}f\lefri{w^{i}} - A^{-1}f\lefri{v^{i}}\right\|_{\infty} \\
&= \left\|A^{-1}\lefri{f\lefri{w^{i}} - f\lefri{v^{i}}}\right\|_{\infty} \\
&\leq \left\|A^{-1}\right\|_{\infty} \left\|f\lefri{w^{i}} - f\lefri{v^{i}}\right\|_{\infty}.
\end{align*}%
Considering \(\left\|f\lefri{w^{i}} - f\lefri{v^{i}}\right\|_{\infty}\), we see that%
\begin{align}
\left\|f\lefri{w^{i}} - f\lefri{v^{i}}\right\|_{\infty} &= \left| \sum_{j=1}^{m} b_{j} \left[ \nabla V\lefri{\sum_{i=1}^{n}w_{k}^{i} \phi_{i}\lefri{c_{j}h}} -\nabla V\lefri{\sum_{i=1}^{n}v_{k}^{i} \phi_{i}\lefri{c_{j}h}}\right]\phi_{p^{*}}\lefri{c_{j}h} \right|, \label{maxele}
\end{align}%
for some appropriate index \(p^{*}\). Note that the first and last terms of \(\MNorm{f\lefri{w^{i}} - f\lefri{v^{i}}}{\infty}\) will vanish, so the maximum element must take the form of (\ref{maxele}). Let \(\phi^{i}\lefri{t} = \lefri{\phi_{1}\lefri{t},\phi_{2}\lefri{t},...,\phi_{n}\lefri{t}}\). Let \(\LipC\) be the Lipschitz constant for \(\nabla V\lefri{q}\). Now
\begin{align*}
\MNorm{f\lefri{w^{i}} - f\lefri{v^{i}}}{\infty} &=\left| h\sum_{j=1}^{m} b_{j} \left[\nabla V\lefri{\sum_{i=1}^{n} w_{k}^{i} \phi_{i}\lefri{c_{j}h}} - \nabla V\lefri{\sum_{i=1}^{n} v_{k}^{i} \phi_{i}\lefri{c_{j}h}}\right]\phi_{p^{*}}\lefri{c_{j}h} \right| \\
&\leq h\sum_{j=1}^{m}\left|b_{j}\right| \left|\left[\nabla V\lefri{\sum_{i=1}^{n} w_{k}^{i} \phi_{i}\lefri{c_{j}h}} - \nabla V\lefri{\sum_{i=1}^{n} v_{k}^{i} \phi_{i}\lefri{c_{j}h}}\right]\right|\left|\phi_{p^{*}}\lefri{c_{j}h} \right| \\
&\leq h\sum_{j=1}^{m} b_{j} \LipC\left|\sum_{i=1}^{n} w_{k}^{i} \phi_{i}\lefri{c_{j}h} - \sum_{i=1}^{n} v_{k}^{i} \phi_{i}\lefri{c_{j}h} \right| \left|\phi_{p^{*}}\lefri{c_{j}h} \right| \\
&= h\sum_{j=1}^{m} b_{j} \LipC\left|\sum_{i=1}^{n} \lefri{w_{k}^{i} - v_{k}^{i}} \phi_{i}\lefri{c_{j}h}\right|\left|\phi_{p^{*}}\lefri{c_{j}h} \right| \\
&\leq h \sum_{j=1}^{m} b_{j} \LipC \left\|w^{i} - v^{i}\right\|_{\infty} \left\| \phi^{i}\lefri{c_{j}h} \right\|_{1} \left|\phi_{p^{*}}\lefri{c_{j}h}\right| \\
&\leq h \sum_{j=1}^{m} b_{j} \LipC \max_{j}\lefri{\left\|\phi^{i}\lefri{c_{j}h}\right\|_{1}\left|\phi_{p^{*}}\lefri{c_{j}h}\right|} \left\|w^{i} - v^{i}\right\|_{\infty}\\
&= h \LipC \max_{j}\lefri{\left\|\phi^{i}\lefri{c_{j}h}\right\|_{1}\left|\phi_{p^{*}}\lefri{c_{j}h}\right|} \left\|w^{i} - v^{i}\right\|_{\infty}.
\end{align*}%
Hence, we derive the inequality%
\begin{align*}
\left\| \Phi\lefri{w^{i}} - \Phi\lefri{v^{i}} \right\|_{\infty} &\leq h \left\|A^{-1}\right\|_{\infty} \LipC \max_{j}\lefri{\left\|\phi^{i}\lefri{c_{j}h}\right\|_{1}\left|\phi_{p^{*}}\lefri{c_{j}h}\right|} \left\|w^{i} - v^{i}\right\|_{\infty},
\end{align*}%
and since by assumption \(\MNorm{A^{-1}}{\infty} \leq \MNorm{A_{1}^{-1}}{\infty}\),
\begin{align*}
\MNorm{\Phi\lefri{w^{i}} - \Phi\lefri{v^{i}}}{\infty} &\leq h \left\|A_{1}^{-1}\right\|_{\infty} \LipC \max_{j}\lefri{\left\|\phi^{i}\lefri{c_{j}h}\right\|_{1}\left|\phi_{p^{*}}\lefri{c_{j}h}\right|} \left\|w^{i} - v^{i}\right\|_{\infty}.
\end{align*}%
Thus if:%
\begin{align*}
h < \lefri{\left\|A_{1}^{-1}\right\|_{\infty} \LipC \max_{j}\lefri{\left\|\phi^{i}\lefri{c_{j}h}\right\|_{1}\left|\phi_{p^{*}}\lefri{c_{j}h}\right|}}^{-1},
\end{align*}
then%
\begin{align*}
\left\|\Phi\lefri{w^{i}} - \Phi\lefri{v^{i}}\right\|_{\infty} \leq k \left\|w^{i} - v^{i} \right\|_{\infty},
\end{align*}%
where \(k < 1\), which establishes that \(\Phi\lefri{\cdot}\) is a contraction mapping, and establishes the existence of a unique fixed point, and thus the existence of unique steps of the one step method.
\end{proof}

A critical assumption made during the proof of existence and uniqueness is that the matrix \(A\) is nonsingular. This property depends on the choice of basis functions \(\phi_{i}\). However, using a polynomial basis, like Lagrange interpolation polynomials, it can be shown that \(A\) is invertible.%
%
%
\begin{lemma}\emph{(\(A\) is invertible)}\label{AInvert} If \(\left\{\phi_{i}\right\}_{i=1}^{n}\) is a polynomial basis of \(P_{n}\), the space of polynomials of degree at most \(n\), M is symmetric positive-definite, and the quadrature rule is order at least \(2n + 1\), then \(A\) defined by \eqref{AMatrix1} -- \eqref{AMatrix3} is invertible.
\end{lemma}
%
%
\begin{proof} We begin by considering the equation:%
\begin{align*}
Aq^{i} = 0.
\end{align*}%
Let \(\galqn\lefri{t} = \sum_{i=1}^{n} q_{k}^{i} \phi_{i}\lefri{t}\). Considering the definition of \(A\), \(Aq^{i} = 0\) holds if and only if the following equations hold:%
\begin{align}
\galqn\lefri{0} &= 0, & \nonumber \\
h\sum_{j=1}^{m}b_{j} M\dgalqn\lefri{c_{j}h}\dot{\phi}_{p}\lefri{c_{j}h} &= 0, & p = 1,...,(n-1). \label{innerproductcond}
\end{align}%
It can easily be seen that \(\left\{\dot{\phi}_{i}\right\}_{i=1}^{n-1}\) is a basis of \(P_{n-1}\). Using the assumption that the quadrature rule is of order at least \(2n-1\) and that \(M\) is symmetric positive-definite, we can see that (\ref{innerproductcond}) implies:%
\begin{align*}
\int_{0}^{h} M\dgalqn\lefri{t}\dot{\phi}_{p}\lefri{t} \dt &= 0, & p = 1,...,(n-1),
\end{align*}%
but,%
\begin{align*}
\int_{0}^{h} M\dgalqn\lefri{t} \dot{\phi}_{i}\lefri{t} \dt = 0
\end{align*}%
implies
\begin{align*}
\left<\dgalqn, \dot{\phi}_{p}\right> = 0,
\end{align*}%
where \(\left<\cdot,\cdot\right>\) is the standard \(L^{2}\) inner product on \(\left[0,h\right]\). Since \(\left\{\dot{\phi}_{i}\right\}_{i=1}^{n-1}\) forms a basis for \(P_{n-1}\), \(\dgalqn \in P_{n-1}\), and \(\left<\cdot,\cdot\right>\) is non-degenerate, this implies that \(\dgalqn\lefri{t} = 0\). Thus,%
\begin{align*}
\galqn\lefri{0} = 0 \\
\dgalqn\lefri{t} = 0
\end{align*}%
which implies that \(\galqn\lefri{t} = 0\) and hence \(q^{i} = 0\). Thus, \(Aq^{i} = 0\) then \(q^{i} = 0\), from which it follows that \(A\) is non-singular.
\end{proof}

Another subtle difficulty is that the matrix \(A\) is a function of \(h\). Since we assumed that \(\left\|A^{-1}\right\|_{\infty}\) is bounded to prove Theorem \ref{ExisUniq}, we must show that for any choice of \(h\), the quantity \(\left\|A^{-1}\right\|_{\infty}\) is bounded. We will do this by establishing \(\left\|A^{-1}\right\|_{\infty} \leq \MNorm{A_{1}^{-1}}{\infty}\), where \(A_{1}\) is \(A\) defined with \(h=1\). By Lemma \ref{AInvert}, we know that \(\MNorm{A_{1}^{-1}}{\infty} < \infty\), which establishes the upper bound for \(\MNorm{A^{-1}}{\infty}\). This argument is easily generalized for a higher upper bound on \(h\).%
%
%
\begin{lemma}\emph{(\(\left\|A^{-1}\right\|_{\infty} \leq \left\|A_{1}^{-1}\right\|_{\infty})\)} For the matrix \(A\) defined by \eqref{AMatrix1} -- \eqref{AMatrix3}, if \(h < 1\), \(\left\|A^{-1}\right\|_{\infty} < \left\|A_{1}^{-1}\right\|_{\infty}\) where \(A_{1}\) is \(A\) defined on the interval \(\left[0,1\right]\). 
\end{lemma}%
%
%
\begin{proof} We begin the proof by examining how \(A\) changes as a function of \(h\). First, let \(\left\{\phi_{i}\right\}_{i=1}^{n}\) be the basis for the interval \(\left[0,1\right]\). Then for the interval \(\left[0,h\right]\), the basis functions are%
\begin{align*}
\phi^{h}_{i} \lefri{t} = \phi_{i}\lefri{\frac{t}{h}}
\end{align*}%
and hence the derivatives are:%
\begin{align*}
\dot{\phi}^{h}_{i}\lefri{t} = \frac{1}{h} \dot{\phi}_{i}\lefri{\frac{t}{h}}.
\end{align*}%
Thus, if \(A_{1}\) is the matrix defined by \eqref{AMatrix1} -- \eqref{AMatrix3} on the interval \(\left[0,1\right]\), then for the interval \(\left[0,h\right]\),
\begin{align*}
A = \left(\begin{array}{cc} 1 & 0\\ 0 & \frac{1}{h}I_{\lefri{n-1}\times \lefri{n-1}}\end{array}\right) A_{1},
\end{align*}%
where \(I_{n\times n}\) is the \(n\times n\) identity matrix. This gives
\begin{align*}
A^{-1} = A^{-1}_{1} \left(\begin{array}{cc} 1 & 0 \\ 0 & h I_{\lefri{n-1} \times \lefri{n-1}}\end{array}\right)
\end{align*}%
which gives%
\begin{align*}
\left\|A^{-1}\right\|_{\infty} = \left\|A^{-1}_{1} \left(\begin{array}{cc} 1 & 0 \\ 0 & h I_{\lefri{n-1} \times \lefri{n-1}}\end{array}\right)\right\|_{\infty} \leq \left\|A^{-1}_{1}\right\|_{\infty} \left\|  \left(\begin{array}{cc} 1 & 0 \\ 0 & h I_{\lefri{n-1} \times \lefri{n-1}}\end{array}\right)\right\|_{\infty} = \left\|A^{-1}_{1}\right\|_{\infty},
\end{align*}%
which proves the statement.
\end{proof}

\subsection{Order Optimal and Geometric Convergence}\label{ConSection}%
To determine the rate of convergence for spectral variational integrators, we will utilize Theorem \ref{MarsConv} and a simple extension of Theorem \ref{MarsConv}:%
\begin{theorem} \emph{(Extension of Theorem \ref{MarsConv} to Geometric Convergence)} \label{MarsConvExt}%
Given a regular Lagrangian \(L\) and corresponding Hamiltonian \(H\), the following are equivalent for a discrete Lagrangian \(L_{d}\lefri{q_{0},q_{1},n}\):%
\begin{enumerate}
\item there exists a positive constant \(K\), where \(K < 1\), such that the discrete Hamiltonian map for \(L_{d}\) has error \(\mathcal{O}\lefri{K^{n}}\),
\item there exists a positive constant \(K\), where \(K < 1\), such that the discrete Legendre transforms of \(L_{d}\) have error \(\mathcal{O}\lefri{K^{n}}\),
\item there exists a positive constant \(K\), where \(K < 1\), such that \(L_{d}\) is equivalent to a discrete Lagrangian with error \(\mathcal{O}\lefri{K^{n}}\).
\end{enumerate}%
\end{theorem}%
This theorem provides a fundamental tool for the analysis of Galerkin variational methods. Its proof is almost identical to that of Theorem \ref{MarsConv}, and can be found in the appendix. The critical result is that the order of the error of the discrete Hamiltonian flow map, from which we construct the discrete flow, has the same order as the discrete Lagrangian from which it is constructed. Thus, in order to determine the order of the error of the flow generated by spectral variational integrators, we need only determine how well the discrete Lagrangian approximates the exact discrete Lagrangian. This is a key result which greatly reduces the difficulty of the error analysis of Galerkin variational integrators.

Naturally, the goal of constructing spectral variational integrators is constructing a variational method that has geometric convergence. To this end, it is essential to establish that Galerkin type integrators inherit the convergence properties of the spaces which are used to construct them. The order optimality result is related to the problem of $\Gamma$-convergence (see, for example, \citet{Da1993}), as the Galerkin discrete Lagrangians are given by extremizers of an approximating sequence of variational problems, and the exact discrete Lagrangian is the extremizer of the limiting variational problem. The $\Gamma$-convergence of variational integrators was studied in \citet{MuOr2004}, and our approach involves a refinement of their analysis. We now state our results, which establish not only the geometric convergence of spectral variational integrators, but also order optimality of all Galerkin variational integrators under appropriate smoothness assumptions.%
%
%
\begin{theorem} \emph{(Order Optimality of Galerkin Variational Integrators)} \label{OptConv}
Given an interval \(\left[0,h\right]\) and a Lagrangian \(L:TQ \rightarrow \mathbb{R}\), let \(\truq\) be the exact solution to the Euler-Lagrange equations subject to the conditions \(\truq\lefri{0} = q_{0}\) and \(\truq\lefri{h} = q_{h}\), and let \(\galqn\) be the stationary point of a Galerkin variational discrete action, i.e. if \(\GDLn:Q\times Q \times \mathbb{R} \rightarrow \mathbb{R}\),
\begin{align*}
\GDLh{q_{0}}{q_{h}}{h} = \ext_{\galargsf{q_{0}}{q_{h}}} \mathbb{S}_{d}\lefri{\left\{q_{i}\right\}_{i=1}^{n}} = \ext_{\galargsf{q_{0}}{q_{h}}}h\sum_{j=1}^{m} b_{j} L\lefri{q_{n}\lefri{c_{j}h},\dot{q}_{n}\lefri{c_{j}h}},
\end{align*}%
then
\begin{align*}
\galqn = \argmin_{\galargsf{q_{0}}{q_{h}}}h\sum_{j=1}^{m} b_{j} L\lefri{q_{n}\lefri{c_{j}h},\dot{q}_{n}\lefri{c_{j}h}}.
\end{align*}
If:%
\begin{enumerate}
\item there exists a constant \(\ApproxC\) independent of \(h\), such that, for each \(h\), there exists a curve \(\optqn \in \FdFSpace\), such that,%
\begin{align*}
  \left|\lefri{\optqn\lefri{t},\doptqn\lefri{t}} - \lefri{\truq\lefri{t},\dtruq\lefri{t}}\right| & \leq  \ApproxC h^{n},
\end{align*}
\item there exists a closed and bounded neighborhood \(U \subset TQ\), such that \(\lefri{\truq\lefri{t},\dtruq\lefri{t}} \in U\), \(\lefri{\optqn\lefri{t},\doptqn\lefri{t}} \in U\) for all \(t\), and all partial derivatives of \(L\) are continuous on \(U\),
\item for the quadrature rule \(\mathcal{G}\lefri{f} = h\sum_{j=1}^{m} b_{j}f\lefri{c_{j}h} \approx \int_{0}^{h}f\lefri{t}\dt\), there exists a constant \(\QuadC\), such that,%
\begin{align*}
  \left|\int_{0}^{h} L\lefri{q_{n}\lefri{t}, \dot{q}_{n}\lefri{t}}\dt - h\sum_{j=1}^{m} b_{j} L\lefri{q_{n}\lefri{c_{j}h},\dot{q}_{n}\lefri{c_{j}h}}\right| \leq \QuadC h^{n+1},
\end{align*}%
for any \(q_{n} \in \mathbb{M}^{n}\lefri{\left[0,h\right],Q}\),
\item and the stationary points \(\truq\), \(\galqn\) minimize their respective actions,%
\end{enumerate}%
then%
\begin{align*}
  \left|\EDLh{q_{0}}{q_{h}}{h} - \GDLh{q_{0}}{q_{h}}{h}\right| \leq \OptC h^{n+1},
\end{align*}%
for some constant \(\OptC\) independent of \(h\), i.e. discrete Lagrangian \(L_{d}\) has error \(\mathcal{O}\lefri{h^{n+1}}\), and hence the discrete Hamiltonian flow map has error \(\mathcal{O}\lefri{h^{n+1}}\).%
\end{theorem}
%
%
\begin{proof}%
First, we rewrite both the exact discrete Lagrangian and the Galerkin discrete Lagrangian:%
\begin{align*}
\left|\EDLh{q_{0}}{q_{h}}{h} - \GDLh{q_{0}}{q_{h}}{h}\right| &= \left|\int_{0}^{h}L\lefri{\truq\lefri{t},\dtruq\lefri{t}}\dt - \mathcal{G}\lefri{L\lefri{\galqn\lefri{t},\dgalqn\lefri{t}}}\right| \\
&= \left|\int_{0}^{h}L\lefri{\truq\lefri{t},\dtruq\lefri{t}}\dt - h\sum_{j=1}^{m}b_{j} L\lefri{\galqn\lefri{c_{j}h},\dgalqn\lefri{c_{j}h}}\right| \\
&= \left|\int_{0}^{h}L\lefri{\truq,\dtruq}\dt - h\sum_{j=1}^{m}b_{j} L\lefri{\galqn,\dgalqn}\right|,
\end{align*}%
where in the last line, we have suppressed the \(t\) argument, a convention we will continue throughout the proof. Now we introduce the action evaluated on the \(\optqn\) curve, which is an approximation with error \(\mathcal{O}\lefri{h^{n}}\) to the exact solution \(\truq\):%
\begin{subequations}
\begin{align}
 \left|\int_{0}^{h}L\lefri{\truq,\dtruq}\dt - h\sum_{j=1}^{m}b_{j} L\lefri{\galqn,\dgalqn}\right| &= 
 \left|\int_{0}^{h}L\lefri{\truq,\dtruq}\dt - \int_{0}^{h} L\lefri{\optqn, \doptqn} \dt \right. \nonumber\\
& \hspace{10em} + \left.\int_{0}^{h} L\lefri{\optqn,\doptqn}\dt - h\sum_{j=1}^{m}b_{j} L\lefri{\galqn,\dgalqn}\right| \nonumber \\
& \leq \left|\int_{0}^{h}L\lefri{\truq,\dtruq}\dt - \int_{0}^{h} L \lefri{\optqn, \doptqn}\dt\right| \label{FirstTerm} \\
& \hspace{10em} +\left|\int_{0}^{h} L \lefri{\optqn,\doptqn} \dt - h\sum_{j=1}^{m}b_{j}L\lefri{\galqn,\dgalqn}\right|. \label{SecondTerm}
\end{align}
\end{subequations}
Considering the first term (\ref{FirstTerm}):%
\begin{align*}
  \left|\int_{0}^{h}L\lefri{\truq,\dtruq}\dt - \int_{0}^{h} L \lefri{\optqn, \doptqn}\dt\right| &= \left|\int_{0}^{h}L\lefri{\truq,\dtruq} - L \lefri{\optqn, \doptqn} \dt\right| \\
& \leq \int_{0}^{h} \left|L\lefri{\truq,\dtruq} - L\lefri{\optqn,\doptqn}\right|\dt. 
\end{align*}%
By assumption, all partials of \(L\) are continuous on \(U\), and since \(U\) is closed and bounded, this implies \(L\) is Lipschitz on \(U\). Let \(\LagLipC\) denote that Lipschitz constant. Since, again by assumption, \(\lefri{\truq,\dtruq} \in U\) and \(\lefri{\optqn,\doptqn} \in U\), we can rewrite:%
\begin{align*}
\int_{0}^{h}\left|L\lefri{\truq,\dtruq} - L \lefri{\optqn,\doptqn}\right|\dt & \leq  \int_{0}^{h} \LagLipC \left|\lefri{\truq,\dtruq} - \lefri{\optqn,\doptqn}\right| \dt \\
&\leq \int_{0}^{h} \LagLipC \ApproxC h^{n} \dt\\
&= \LagLipC \ApproxC h^{n+1}, 
\end{align*}%
where we have made use of the best approximation estimate. Hence,%
\begin{align}
\left|\int_{0}^{h} L\lefri{\truq, \dtruq} \dt - \int_{0}^{h} L\lefri{\optqn,\doptqn} \dt\right| \leq L_{\alpha}C_{1}h^{n+1}. \label{FirstTermIneq}
\end{align}%
Next, considering the second term (\ref{SecondTerm}),%
\begin{align*}
\left|\int_{0}^{h} L\lefri{\optqn,\doptqn} \dt - h\sum_{j=1}^{m} b_{j} L\lefri{\galqn,\dgalqn} \right|,
\end{align*}%
since \(\galqn\), the stationary point of the discrete action, minimizes its action and \(\optqn \in \mathbb{M}^{n}\lefri{\left[0,h\right],Q}\),%
\begin{align}
h\sum_{j=1}^{m} b_{j} L\lefri{\galqn,\dgalqn} \leq h\sum_{j=1}^{m} b_{j} L \lefri{\optqn,\doptqn} \leq \int_{0}^{h} L\lefri{\optqn,\doptqn} \dt + \QuadC h^{n+1} \label{UpperBound}
\end{align}%
where the inequalities follow from the assumptions on the order of the quadrature rule. Furthermore,%
\begin{align}
h\sum_{j=1}^{m} b_{j} L \lefri{\galqn,\dgalqn} &\geq \int_{0}^{h} L\lefri{\galqn,\dgalqn} \dt - \QuadC h^{n+1} \nonumber \\
&\geq \int_{0}^{h} L\lefri{\truq,\dtruq}  \dt - \QuadC h^{n+1} \nonumber\\
&\geq \int_{0}^{h} L\lefri{\optqn,\doptqn} \dt - \LagLipC \ApproxC h^{n+1} - \QuadC h^{n+1}, \label{LowerBound}
\end{align}%
where the inequalities follow from (\ref{FirstTermIneq}), the order of the quadrature rule, and the assumption that \(\truq\) minimizes its action. Putting (\ref{UpperBound}) and (\ref{LowerBound}) together, we can conclude:%
\begin{align}
\left|\int_{0}^{h}L\lefri{\optqn,\doptqn}\dt - h\sum_{j=1}^{m} b_{j} L\lefri{\galqn,\dgalqn}\right| \leq \lefri{\LagLipC \ApproxC + \QuadC}h^{n+1} \label{SecondTermIneq}.
\end{align}%
Now, combining the bounds (\ref{FirstTermIneq}) and (\ref{SecondTermIneq}) in (\ref{FirstTerm}) and (\ref{SecondTerm}), we can conclude%
\begin{align*}
\left|\EDLh{q_{0}}{q_{h}}{h} - \GDLh{q_{0}}{q_{h}}{h}\right| \leq \lefri{2\LagLipC \ApproxC + \QuadC}h^{n+1}
\end{align*}
which, combined with Theorem \ref{MarsConv}, establishes the order of the error of the integrator.
\end{proof}%
%
%
The above proof establishes a significant convergence result for Galerkin variational integrators, namely that for sufficiently well behaved Lagrangians, Galerkin variational integrators will produce discrete approximate flows that converge to the exact flow as \(h\rightarrow0\) with the highest possible order allowed by the approximation space, provided the quadrature rule is of sufficiently high order.

We will discuss assumption 4 in \S \ref{MinAction}. While in general we cannot assume that stationary points of the action are minimizers, it can be shown that for Lagrangians of the canonical form%
\begin{align*}
L\lefri{q,\dot{q}} = \dot{q}^{T}M\dot{q} - V\lefri{q},
\end{align*}%
under some mild assumptions on the derivatives of \(V\) and the accuracy of the quadrature rule, there always exists an interval \(\left[0,h\right]\) over which stationary points are minimizers. In \S \ref{MinAction} we will show the result extends to the discretized action of Galerkin variational integrators. A similar result was established in \citet{MuOr2004}.

Geometric convergence of spectral variational integrators is not strictly covered under the proof of order optimality. While the above theorem establishes convergence of Galerkin variational integrators by shrinking \(h\), the interval length of each discrete Lagrangian, spectral variational integrators achieve convergence by holding the interval length of each discrete Lagrangian constant and increasing the dimension of the approximation space \(\FdFSpace\). Thus, for spectral variational integrators, we have the following analogous convergence theorem: %
%
%
\begin{theorem} \emph{(Geometric Convergence of Spectral Variational Integrators)} \label{SpecConv}
Given an interval \(\left[0,h\right]\) and a Lagrangian \(L:TQ \rightarrow \mathbb{R}\), let \(\truq\) be the exact solution to the Euler-Lagrange equations, and \(\galqn\) be the stationary point of the spectral variational discrete action:
\begin{align*}
\SDLN{q_{0}}{q_{h}}{n} = \ext_{\galargsf{q_{0}}{q_{h}}} \mathbb{S}_{d}\lefri{\left\{q_{i}\right\}_{i=1}^{n}}= \ext_{\galargsf{q_{0}}{q_{h}}} h\sum_{j=0}^{m_{n}} \bnj L\lefri{q_{n}\lefri{\cnjh},\dot{q}_{n}\lefri{\cnjh}}.
\end{align*}
If:%
\begin{enumerate}
\item there exists constants \(\ApproxC,\ApproxK\), \(\ApproxK < 1\), independent of \(n\) such that, for each \(n\), there exists a curve \(\optqn \in \mathbb{M}^{n}\lefri{\left[0,h\right],Q}\), such that,%
\begin{align*}
  \left|\lefri{\truq,\dtruq} - \lefri{\optqn,\doptqn}\right| & \leq  \ApproxC \ApproxK^{n},
\end{align*}
\item there exists a closed and bounded neighborhood \(U \subset TQ\), such that, \(\lefri{\truq\lefri{t},\dtruq\lefri{t}} \in U\) and \(\lefri{\optqn\lefri{t},\doptqn\lefri{t}} \in U\) for all \(t\) and \(n\), and all partial derivatives of \(L\) are continuous on \(U\),
\item for the sequence of quadrature rules \(\mathcal{G}_{n}\lefri{f} = \sum_{j=1}^{m_{n}} \bnj f\lefri{\cnjh} \approx \int_{0}^{h}f\lefri{t}\dt\), there exists constants \(\QuadC\), \(\QuadK\), \(\QuadK < 1\), independent of \(n\) such that%
\begin{align*}
  \left|\int_{0}^{h} L\lefri{q_{n}\lefri{t}, \dot{q}_{n}\lefri{t}}\dt - h\sum_{j=1}^{m_{n}} \bnj L\lefri{q_{n}\lefri{\cnjh},\dot{q}_{n}\lefri{\cnjh}}\right| \leq \QuadC \QuadK^{n},
\end{align*}%
for any \(q_{n}\in \mathbb{M}^{n}\lefri{\left[0,h\right],Q}\),
\item and the stationary points \(\truq\), \(\galqn\) minimize their respective actions,%
\end{enumerate}%
then%
\begin{align}
  \left|\EDL{q_{0}}{q_{1}} - \SDLN{q_{0}}{q_{1}}{n}\right| \leq \SpecC\SpecK^{n} \label{GeoBound}
\end{align}%
for some constants \(\SpecC,\SpecK\), \(\SpecK < 1\), independent of \(n\), and hence the discrete Hamiltonian flow map has error \(\mathcal{O}\lefri{\SpecK^{n}}\).%
\end{theorem}
The proof of the above theorem is very similar to that of order optimality, and would be tedious to repeat here. It can be found in the appendix. The main differences between the proofs are the assumption of the sequence of converging functions in the increasingly high-dimensional approximation spaces, and the assumption of a sequence of increasingly high-order quadrature rules. These assumptions are used in the obvious way in the modified proof. 

\subsection{Minimization of the Action} \label{MinAction}
One of the major assumptions made in the convergence theorems (\ref{OptConv}) and (\ref{SpecConv}) is that the the stationary points of both the continuous and discrete actions are minimizers over the interval \(\left[0,h\right]\). This type of minimization requirement is similar to the one made in the paper on \(\Gamma\)-convergence of variational integrators by \citet{MuOr2004}. In fact, the results in \citet{MuOr2004} can easily be extended to demonstrate that for sufficiently well-behaved Lagrangians of the form%
\begin{align*}
L\lefri{q,\dot{q}} = \frac{1}{2}\dot{q}^{T}M\dot{q} - V\lefri{q},
\end{align*}%
where \(q \in \CQ\), there exists an interval \(\left[0,h\right]\), such that stationary points of the Galerkin action are minimizers.\\
%
%
\begin{theorem}%
Consider a Lagrangian of the form%
\begin{align*}
L\lefri{q,\dot{q}} = \frac{1}{2}\dot{q}^{T}M\dot{q} - V\lefri{q}
\end{align*}%
where \(q \in \CQ\) and each component \(q^{d}\) of \(q\), \(q^{d} \in \CQ\), is a polynomial of degree at most \(s\). Assume \(M\) is symmetric positive-definite and all second-order partial derivatives of \(V\) exist, and are continuous and bounded. Then, there exists a time interval \(\left[0,h\right]\) such that stationary points of the discrete action,%
\begin{align*}
\mathbb{S}_{d}\lefri{\left\{q_{k}^{i}\right\}_{i=1}^{n}} = h\sum_{j=1}^{m} b_{j} \lefri{\frac{1}{2} \dgalqn\lefri{c_{j}h}^{T}M\dgalqn\lefri{c_{j}h} - V\lefri{\galqn\lefri{c_{j}h}}},
\end{align*}%
on this time interval are minimizers if the quadrature rule used to construct the discrete action is of order at least \(2s+1\).%
\end{theorem}%
We quickly note that the assumption that each component of \(q\), \(q^{d}\), is a polynomial of degree \emph{at most} \(s\) allows for discretizations where different components of the configuration space are discretized with polynomials of different degrees. This allows for more efficient discretizations where slower evolving components are discretized with lower-degree polynomials than faster evolving ones.
%
%
\begin{proof}%
Let \(\galqn\) be a stationary point of the discrete action \(\mathbb{S}_{d}\lefri{\cdot}\), and let \(\delta q\) be an arbitrary perturbation of the stationary point \(\galqn\), under the conditions \(\delta q^{d} \in P_{S_{d}}\), \(\delta q\lefri{0} = \delta q\lefri{h} = 0\), which is uniquely defined by \(\left\{\delta q_{k}^{i}\right\}_{i=1}^{n} \subset Q\). Then,%
\begin{align*}
&\mathbb{S}_{d}\lefri{\left\{q_{k}^{i} + \delta q_{k}^{i}\right\}_{i=1}^{n}} - \mathbb{S}_{d}\lefri{\left\{q_{k}^{i}\right\}_{i=1}^{n}} \\
&= h\sum_{j}^{m} b_{j} \lefri{\frac{1}{2} \lefri{\dgalqn + \delta \dot{q}}^{T}M\lefri{\dgalqn + \delta \dot{q}} - V\lefri{\galqn + \delta q}} - h\sum_{j}^{m} b_{j} \lefri{\frac{1}{2} \dgalqn^{T}M\dgalqn - V\lefri{\galqn}}\\
&= h\sum_{j}^{m} b_{j} \lefri{\frac{1}{2} \lefri{\dgalqn + \delta \dot{q}}^{T}M\lefri{\dgalqn + \delta \dot{q}} - V\lefri{\galqn + \delta q} - \frac{1}{2} \dgalqn^{T}M\dgalqn + V\lefri{\galqn}}.
\end{align*}%
Making use of Taylor's remainder theorem, we expand:%
\begin{align*}
 V\lefri{\galqn + \delta q} = V\lefri{\galqn} + \nabla V\lefri{\galqn}\cdot \delta q + \frac{1}{2}\delta \galqn^{T} R \delta \galqn,
\end{align*}%
where \(\left|R_{lm}\right| \leq \sup_{l,m}\left|\frac{\partial^{2}V}{\partial q_{l} \partial q_{m}}\right|\). Using this expansion, we rewrite%
\begin{align*}
\mathbb{S}_{d}\lefri{\left\{q_{k}^{i} + \delta q_{k}^{i}\right\}_{i=1}^{n}} - \mathbb{S}_{d}\lefri{\left\{q_{k}^{i}\right\}_{i=1}^{n}} &= h\sum_{j}^{m} b_{j} \left(\frac{1}{2} \lefri{\dgalqn + \delta \dot{q}}^{T}M\lefri{\dgalqn + \delta \dot{q}} - V\lefri{\galqn} - \nabla V\lefri{\galqn} \cdot \delta q \right. \\
& \hspace{5em} -\left. \frac{1}{2} \delta q^{T} R \delta q - \lefri{\frac{1}{2} \dgalqn^{T}M\dgalqn + V\lefri{\galqn}}\right)
\end{align*}%
which, given the symmetry in \(M\), rearranges to:%
\begin{align*}
\mathbb{S}_{d}\lefri{\left\{q_{k}^{i} + \delta q_{k}^{i}\right\}_{i=1}^{n}} - \mathbb{S}_{d}\lefri{\left\{q_{k}^{i}\right\}_{i=1}^{n}} &= h\sum_{j}^{m} b_{j} \lefri{\dgalqn^{T}M\delta\dot{q} - \nabla V\lefri{\galqn} \cdot \delta q + \frac{1}{2} \delta \dot{q}^{T}M\delta \dot{q}  - \frac{1}{2} \delta q^{T} R \delta q}.
\end{align*}%
Now, it should be noted that the stationarity condition for the discrete Euler-Lagrange equations is%
\begin{align*}
h\sum_{j=1}^{m} b_{j}\lefri{\dgalqn^{T}M\delta \dot{q} - \nabla V\lefri{\galqn} \cdot \delta q} = 0
\end{align*}%
for arbitrary \(\delta q\), which allows us to simplify the expression to%
\begin{align*}
 \mathbb{S}_{d}\lefri{\left\{q_{k}^{i} + \delta q_{k}^{i}\right\}_{i=1}^{n}} - \mathbb{S}_{d}\lefri{\left\{q_{k}^{i}\right\}_{i=1}^{n}} = h\sum_{j}^{m} b_{j} \lefri{\frac{1}{2} \delta \dot{q}^{T}M\delta \dot{q}  - \frac{1}{2} \delta q^{T} R \delta q}.
\end{align*}%
Now, using the assumption that the partial derivatives of \(V\) are bounded, \(\left|R_{lm}\right| \leq \left|\frac{\partial^{2} V}{\partial q_{l}\partial q_{m}}\right| < C_{R}\), and standard matrix inequalities, we get the inequality:%
\begin{align} 
\delta q^{T} R \delta q \leq \left\|R\delta q\right\|_{2} \left\|\delta q\right\|_{2} \leq \left\|R\right\|_{2} \left\|\delta q\right\|_{2}^{2} \leq \left\|R\right\|_{F} \left\|\delta q\right\|^{2}_{2} \leq DC_{R}\left\|\delta q\right\|^{2}_{2} = DC_{R}\delta q^{T}\delta q, \label{HBound}
\end{align}%
where \(D\) is the number of spatial dimensions of \(Q\). Thus%
\begin{align*}
h\sum_{j}^{m} b_{j} \lefri{\frac{1}{2} \delta \dot{q}^{T}M\delta \dot{q}  - \frac{1}{2} \delta q^{T} R \delta q} \geq  h\sum_{j}^{m} b_{j} \lefri{\frac{1}{2} \delta \dot{q}^{T}M\delta \dot{q}  - \frac{1}{2} DC_{R}\delta q^{T}\delta q}.
\end{align*}%
Because \(M\) is symmetric positive-definite, there exists \(m > 0\) such that \(x^{T}Mx \geq mx^{T}x\) for any \(x\). Hence, %
\begin{align*}
h\sum_{j}^{m} b_{j} \lefri{\frac{1}{2} \delta \dot{q}^{T}M\delta \dot{q}  - \frac{1}{2} DC_{R}\delta q^{T}\delta q} \geq h\sum_{j}^{m} b_{j} \lefri{\frac{1}{2} m\delta\dot{q}^{T}\delta\dot{q}  - \frac{1}{2} DC_{R}\delta q^{T}\delta q}.
\end{align*}%
Now, we note that since each component of \(\delta q\) is a polynomial of degree at most \(s\), \(\delta q^{T} \delta q\) and \(\delta \dot{q}^{T} \delta \dot{q}\) are both polynomials of degree less than or equal to \(2s\). Since our quadrature rule is of order \(2s+1\), the quadrature rule is exact, and we can rewrite%
\begin{align*}
h\sum_{j}^{m} b_{j} \lefri{\frac{1}{2} m\delta \dot{q}^{T}\delta \dot{q}  - \frac{1}{2} DC_{R} \delta q^{T} \delta q} &= \frac{1}{2}\int_{0}^{h} m \delta \dot{q}^{T}\delta \dot{q} - DC_{R} \delta q^{T} \delta q\dt \\
  &= \frac{1}{2}\lefri{\int_{0}^{h} m \delta \dot{q}^{T}\delta \dot{q} \dt - \int_{0}^{h} DC_{R} \delta q^{T} \delta q \dt}. 
\end{align*}%
From here, we note that \(\delta q \in \HoQ\), and make use of the Poincar\'{e} inequality to conclude%
\begin{align*}
\frac{1}{2}\lefri{\int_{0}^{h}m \delta \dot{q}^{T}\delta \dot{q} \dt - \int_{0}^{h} nC_{R} \delta q^{T} \delta q \dt} & \geq \frac{1}{2} \lefri{m\frac{\pi^{2}}{h^{2}}\int_{0}^{h} \delta q^{T}\delta q \dt - DC_{R}\int_{0}^{h} \delta q^{T} \delta q \dt} \\
  &= \frac{1}{2}\lefri{\frac{m\pi^{2}}{h^{2}} - DC_{R}} \int_{0}^{h} \delta q^{T} \delta q \dt. 
\end{align*}%
Since \(\int_{0}^{h} \delta q^{T} \delta q \dt > 0\), 
\begin{eqnarray*}
 \mathbb{S}_{d}\lefri{\left\{q_{k}^{i} + \delta q_{k}^{i}\right\}_{i=1}^{n}} - \mathbb{S}_{d}\lefri{\left\{q_{k}^{i}\right\}_{i=1}^{n}} \geq \frac{1}{2}\lefri{\frac{m\pi^{2}}{h^{2}} - DC_{R}} \int_{0}^{h} \delta q^{T} \delta q \dt > 0
\end{eqnarray*}
so long as \(h < \sqrt{\frac{m \pi^{2}}{DC_{R}}}\).
\end{proof}

\subsection{Convergence of Galerkin Curves and Noether Quantities}
 
\subsubsection{Galerkin Curves}

In order to construct the one-step method, spectral variational integrators determine a curve,%
\begin{align*}
\galqn\lefri{t} = \sum_{i=1}^{n}q^{i}_{k}\phi_{i}\lefri{t},
\end{align*}%
which satisfies%
\begin{align*}
\galqn\lefri{t} &= \argmin_{\galargsf{q_{k}}{q_{k+1}}} h\sum_{j=1}^{m}b_{j}L\lefri{\galqn\lefri{c_{j}h},\dgalqn\lefri{c_{j}h}}.
\end{align*}%
Evaluating this curve at \(h\) defines the next step of the one-step method, \(q_{k+1} = \galqn\lefri{h}\), but the curve itself has many desirable properties which makes it a good continuous approximation to the true solution of the Euler Lagrange equations \(\truq\lefri{t}\). In this section, we will examine some of the favorable properties of \(\galqn\lefri{t}\), hereafter referred to as the \emph{Galerkin curve}.

However, before discussing the properties of the Galerkin curve, it is useful review the different curves with which we are working. We have already defined the Galerkin curve, \(\galqn\lefri{t}\), and we will also be making use of the local solution to the Euler-Lagrange equations \(\truq\lefri{t}\), where%
\begin{align*}
\truq\lefri{t} = \argmin_{\truqargsf{q_{k}}{q_{k+1}}} \int_{0}^{h}L\lefri{q\lefri{t},\dot{q}\lefri{t}}\dt.
\end{align*}%
However, while for each interval \(\truq\) satisfies the Euler-Lagrange equations exactly, it is not the exact solution of the Euler-Lagrange equations globally, as \(q_{k} \neq \Phi_{kh}\lefri{q_{0},\dot{q}_{0}}\), where \(\Phi_{t}\lefri{q_{0},\dot{q}_{0}}\) is the flow of the Euler-Lagrange vector field. This is particularly important when discussing invariants, where the invariants of \(\truq\) remain constant within a time-step, but not from time-step to time-step. 

The first property of the Galerkin curve that we will examine is its rate of convergence to the true flow of the Euler-Lagrange vector field. There are two general sources of error that affect the convergence of these curves, the first being the accuracy to which the curves approximate the local solution to the Euler-Lagrange equations over the interval \(\left[0,h\right]\) with the boundary \(\lefri{q_{k},q_{k+1}}\), and the second being the accuracy of the boundary conditions \(\lefri{q_{k},q_{k+1}}\) as approximations to a true sampling of the exact flow. Numerical experiments will show that often the second source of error dominates the first, causing the Galerkin curves to converge at the same rate as the one-step map. However, the accuracy to which the Galerkin curves approximate the true minimizers independent of the error of the boundary can also be established under appropriate assumptions about the action. Two theorems which establish this convergence are presented below.

Before we state the theorems, we quickly recall the definitions of the \emph{Sobolev Norm} \(\SobNorm{\cdot}{p}\),%
\begin{align*}
\SobNorm{f}{p} = \lefri{\LNorm{f}{p}^{p} + \LNorm{\dot{f}}{p}^{p}}^{\frac{1}{p}} = \lefri{\int_{0}^{h}\left|f\right|^{p}\dt + \int_{0}^{h}\left|\dot{f}\right|^{p}\dt}^{\frac{1}{p}}.
\end{align*}%
We will make extensive use of this norm when examining convergence of Galerkin curves.

\begin{theorem}\emph{(Geometric Convergence of Galerkin Curves with \(n\)-Refinement)}\label{GeoGalerk} Under the same assumptions as Theorem \ref{SpecConv}, if at \(\truq\), the action is twice Frechet differentiable, and if the second Frechet derivative of the action \(\mbox{D}^{2}\mathfrak{S}\lefri{\cdot}\left[\cdot,\cdot\right]\) is coercive in a neighborhood \(U\) of \(\truq\), that is,%
\begin{align*}
D^{2}\mathfrak{S}\lefri{\nu}\left[\delta q, \delta q\right] \geq \CoerC\SobNorm{\delta q}{1}^{2},
\end{align*}
for all curves \(\delta q \in \HoQ \) and all \(\nu \in U\), then the curves which minimize the discrete action converge to the true solution geometrically with \(n\)-refinement with respect to \(\SobNorm{\cdot}{1}\). Specifically, if the discrete Hamiltonian flow map has error \(\mathcal{O}\lefri{\SpecK^{n}}\), \(\SpecK < 1\), then the Galerkin curves have error \(\mathcal{O}\lefri{{\sqrt{\SpecK}}^{n}}\).
\end{theorem}%
\begin{proof}
We start with the bound (\ref{GeoBound}) given at the end of Theorem \ref{SpecConv},%
\begin{align*}
\left|\EDL{q_{k}}{q_{k+1}} - \SDLN{q_{k}}{q_{k+1}}{n}\right| \leq \SpecC\SpecK^{n}
\end{align*}
and expand using the definitions of \(\EDL{q_{k}}{q_{k+1}}\) and \(\SDLN{q_{k}}{q_{k+1}}{n}\), as well as the assumed accuracy of the quadrature rule \(\mathcal{G}_{n}\) to derive%
\begin{align}
\SpecC\SpecK^{n} &\geq \left|\EDL{q_{k}}{q_{k+1}} - \SDLN{q_{k}}{q_{k+1}}{n}\right| \label{ReconIneq1} \\
&= \left|\int_{0}^{h}L\lefri{\truq,\dtruq}\dt - h\sum_{j=1}^{m_{n}} \bnj L\lefri{\galqn\lefri{\cnjh},\galqn\lefri{\cnjh}}\right| \nonumber \\
&\geq \left|\int_{0}^{h}L\lefri{\truq,\dtruq}\dt - \int_{0}^{h} L\lefri{\galqn,\galqn}\dt\right| - \QuadC\QuadK^{n} \label{ReconIneq2}\\
& = \left| \mathfrak{S}\lefri{\galqn} - \mathfrak{S}\lefri{\truq} \right| - \QuadC \QuadK^{n}\nonumber
\end{align}
which implies:%
\begin{align*}
\lefri{\SpecC + \QuadC}\SpecK^{n} \geq& \left| \mathfrak{S}\lefri{\galqn} - \mathfrak{S}\lefri{\truq} \right|
\end{align*}%
because \(\SpecK \geq \QuadK\), (see the proof of Theorem \ref{SpecConv} in the appendix). Using this inequality, we make use of a Taylor expansion of \(\mathfrak{S}\lefri{\galqn}\),%
\begin{align*}
\mathfrak{S}\lefri{\galqn} = \mathfrak{S}\lefri{\truq}  + \mbox{D}\mathfrak{S}\lefri{\truq}\left[\galqn - \truq\right] + \frac{1}{2}\mbox{D}^{2}\mathfrak{S}\lefri{\nu}\left[\galqn - \truq,\galqn - \truq\right],
\end{align*}%
for some \(\nu \in U\), to see that
\begin{align*}
\lefri{\SpecC + \QuadC}\SpecK^{n} &\geq \left| \mathfrak{S}\lefri{\galqn} - \mathfrak{S}\lefri{\truq} \right| \\
&= \left|\mathfrak{S}\lefri{\truq} + \mbox{D}\mathfrak{S}\lefri{\truq}\left[\galqn - \truq\right] + \frac{1}{2}\mbox{D}^{2}\mathfrak{S}\lefri{\truq}\left[\galqn - \truq, \galqn - \truq \right]  - \mathfrak{S}\lefri{\truq}\right|.
\end{align*}%
But%
\begin{align*}
D\mathfrak{S}\lefri{\truq}\left[\galqn -\truq\right] &= \int_{0}^{h}\dLdq\lefri{\truq,\dtruq}\lefri{\galqn - \truq} + \dLddq\lefri{\truq,\dtruq}\lefri{\dgalqn - \dtruq} \dt \\
&= \int_{0}^{h} \lefri{\dLdq\lefri{\truq,\dtruq} - \frac{\mbox{d}}{\dt} \dLddq\lefri{\truq,\dtruq}} \cdot \lefri{\galqn - \truq} \dt\\
&= 0,
\end{align*}%
because \(\galqn\lefri{0} = \truq\lefri{0}\) and \(\galqn\lefri{h} = \truq\lefri{h}\) by definition (note that this implies \(\lefri{\galqn - \truq} \in H^{1}_{0}\lefri{\left[0,h\right],Q}\)). Then
\begin{align*}
\lefri{\SpecC + \QuadC}\SpecK^{n} &\geq \left|\mbox{D}^{2}\mathfrak{S}\lefri{\nu}\left[\galqn - \truq,\galqn - \truq \right] \right| \\
&\geq  \CoerC \SobNorm{\galqn - \truq}{1}^{2} \\
C\sqrt{\SpecK}^{n} &\geq \SobNorm{\galqn - \truq}{1}^{2}
\end{align*}%
where \(C = \frac{\SpecC + \QuadC}{\CoerC}\).\end{proof}%
This result shows that Galerkin curves converge to the true solution geometrically with \(n\)-refinement, albeit with a larger geometric constant, and hence a slower rate. By simply replacing the bounds (\ref{ReconIneq1}) and (\ref{ReconIneq2}) from Theorem \ref{SpecConv} with those from Theorem \ref{OptConv} and the term \(\SpecC\SpecK^{n}\) with \(\OptC h^{p}\), an identical argument shows that Galerkin curves converge at half the optimal rate with \(h\)-refinement.%
\begin{theorem}\emph{(Convergence of Galerkin Curves with \(h\)-Refinement)}\label{OptGalerk} Under the same assumptions as Theorem \ref{OptConv}, if at \(\truq\), the action is twice Frechet differentiable, and if the second Frechet derivative of the action \(\mbox{D}^{2}\mathcal{S}\lefri{\cdot}\left[\cdot,\cdot\right]\) is coercive with a constant \(\CoerC\) independent of \(h\) in a neighborhood \(U\) of \(\truq\), for all curves \(\delta q \in H^{1}_{0}\lefri{\left[0,h\right],Q}\), then if the discrete Lagrange map has error \(\mathcal{O}\lefri{h^{p+1}}\), the Galerkin curves have error at most \(\mathcal{O}\lefri{h^{\frac{p+1}{2}}}\) in \(\SobNorm{\cdot}{1}\). If \(\CoerC\) is a function of \(h\), this bound becomes \(\mathcal{O}\lefri{\CoerC\lefri{h}^{-1}h^{\frac{p+1}{2}}}\).
\end{theorem}%
Like the requirement that the stationary points of the actions are minimizers, the requirement that the second Frechet derivative of the action is coercive may appear quite strong at first. Again, the coercivity will depend on the properties of the Lagrangian \(L\), but we can establish that for Lagrangians of the canonical form,%
\begin{align*}
L\lefri{q,\dot{q}} = \frac{1}{2}\dot{q}^{T}M\dot{q} - V\lefri{q}, 
\end{align*}%
there exists a time step \(\left[0,h\right]\) over which the action is coercive on \(\HoQ\).
%
%
\begin{theorem}\emph{(Coercivity of the Action)} For Lagrangian of the form%
\begin{align*}
L\lefri{q,\dot{q}} &= \frac{1}{2}\dot{q}^{T}M\dot{q} - V\lefri{q},
\end{align*}%
where \(M\) is symmetric positive-definite, and the second derivatives of \(V\lefri{q}\) are bounded, there exists an interval \(\left[0,h\right]\) over which the action is coercive over \(\HoQ\), that is,%
\begin{align*}
\mbox{D}^{2}\mathfrak{S}\lefri{\nu}\left[\delta q, \delta q\right] \geq \CoerC \SobNorm{\delta q}{1}^{2},
\end{align*}%
for any \(\delta q \in \HoQ\) and any \(\nu \in \CQ\).
\end{theorem}
%
%
\begin{proof}
First, we note that if%
\begin{align*}
\mathfrak{S}\lefri{\nu} &= \int_{0}^{h} \frac{1}{2}\dot{\nu}^{T}M\dot{\nu} - V\lefri{\nu},
\end{align*}%
then%
\begin{align*}
\mbox{D}^{2}\mathfrak{S}\lefri{\nu}\left[\delta q, \delta q\right] &= \int_{0}^{h}  \delta\dot{q}^{T}M\delta\dot{q} - \delta q^{T}H\lefri{\nu}\delta q \dt\\
&= \int_{0}^{h} \delta \dot{q}^{T}M\delta\dot{q} \dt - \int_{0}^{h} \delta q^{T}H\lefri{\nu} \delta q \dt
\end{align*}%
where \(H\lefri{\nu}\) is the Hessian of \(V\lefri{\nu}\) at the point \(\nu\). Since \(M\) is symmetric positive-definite, and the second derivatives of \(V\lefri{\cdot}\) are bounded, then there exists \(C_{r}\) and \(m\) such that:%
\begin{align}
\int_{0}^{h} \delta \dot{q}^{T}M\delta \dot{q} \dt &\geq \int_{0}^{h}m\delta\dot{q}^{T}\delta\dot{q} \dt \nonumber \\
\int_{0}^{h} \delta q^{T}H\lefri{\nu}\delta q \dt &\leq \int_{0}^{h}DC_{r}\delta q^{T} \delta q \dt, \label{SecondCoerIneq}
\end{align}%
(see (\ref{HBound}) for a derivation of (\ref{SecondCoerIneq})). Hence,%
\begin{align}
\mbox{D}^{2}\mathfrak{S}\lefri{\nu}\left[\delta q, \delta q\right] &\geq \int_{0}^{h} m\delta\dot{q}^{T}\delta\dot{q} \dt - \int_{0}^{h} DC_{r}f^{T}f \dt \nonumber \\
&= \frac{1}{2}m\int_{0}^{h} \delta\dot{q}^{T}\delta\dot{q} \dt + \frac{1}{2}m \int_{0}^{h} \delta\dot{q}^{T}\delta\dot{q} \dt - DC_{r}\int_{0}^{h} \delta q^{T}\delta q\dt. \label{CoerTwoTerms}
\end{align}%
Considering the last two terms in (\ref{CoerTwoTerms}), and noting that \(\delta q \in \HoQ\), we make use of the Poincar\'{e} inequality to derive: %
\begin{align}
\frac{1}{2}m \int_{0}^{h} \delta\dot{q}^{T}\delta\dot{q} \dt - DC_{r}\int_{0}^{h} \delta q^{T}\delta q \dt &\geq \frac{m \pi^{2}}{2h^{2}} \int_{0}^{h} \delta q^{T}\delta q \dt - nC_{r}\int_{0}^{h}\delta q^{T}\delta q \dt \nonumber \\
&\geq \lefri{\frac{m \pi^{2}}{2h^{2}} - DC_{r}} \int_{0}^{h}\delta q^{T}\delta q\dt. \label{CoerPCI}
\end{align}%
Thus, substituting (\ref{CoerPCI}) in for the last two terms of (\ref{CoerTwoTerms}), we conclude:%
\begin{align*}
 \mbox{D}^{2}\mathfrak{S}\lefri{q,\dot{q}}\left[\delta q,\delta q\right] &\geq \lefri{\frac{m \pi^{2}}{2h^{2}} - DC_{r}} \int_{0}^{h} \delta q^{T}\delta q\dt + \frac{m}{2} \int_{0}^{h} \delta\dot{q}^{T}\delta\dot{q} \dt \\
&\geq \min\lefri{\frac{m}{2},\lefri{\frac{m\pi^{2}}{2h^{2}} - DC_{r}}}\lefri{\int_{0}^{h}\delta q^{T}\delta q\dt + \int_{0}^{h}\delta\dot{q}^{T}\delta\dot{q} \dt}\\
&= \min\lefri{\frac{m}{2},\lefri{\frac{m\pi^{2}}{2h^{2}} - DC_{r}}}\lefri{\LNorm{\delta q}{2}^{2} + \LNorm{\delta \dot{q}}{2}^{2}},
\end{align*}%
and making use of H\"{o}lder's inequality, we see that \(\LNorm{\delta q}{2} \geq h^{\frac{1}{2}}\LNorm{\delta q}{1}\), thus%
\begin{align*}
\mbox{D}^{2}\mathfrak{S}\lefri{q,\dot{q}}\left[\delta q,\delta q\right]  &\geq \min\lefri{\frac{m}{2},\lefri{\frac{m\pi^{2}}{2h^{2}} - DC_{r}}}\lefri{h\LNorm{\delta q}{1}^{2} + h\LNorm{\delta \dot{q}}{1}^{2}}\\
&\geq \min\lefri{\frac{mh}{2},\lefri{\frac{m\pi^{2}}{2h} - hDC_{r}}}\frac{1}{2}\lefri{\LNorm{\delta q}{1} + \LNorm{\delta \dot{q}}{1}}^{2}\\
&= \min\lefri{\frac{mh}{4},\lefri{\frac{m\pi^{2}}{4h} - hDC_{r}}}\SobNorm{\delta q}{1}^{2}
\end{align*}%
which establishes the coercivity result.
\end{proof}%
\subsubsection{Noether Quantities}%

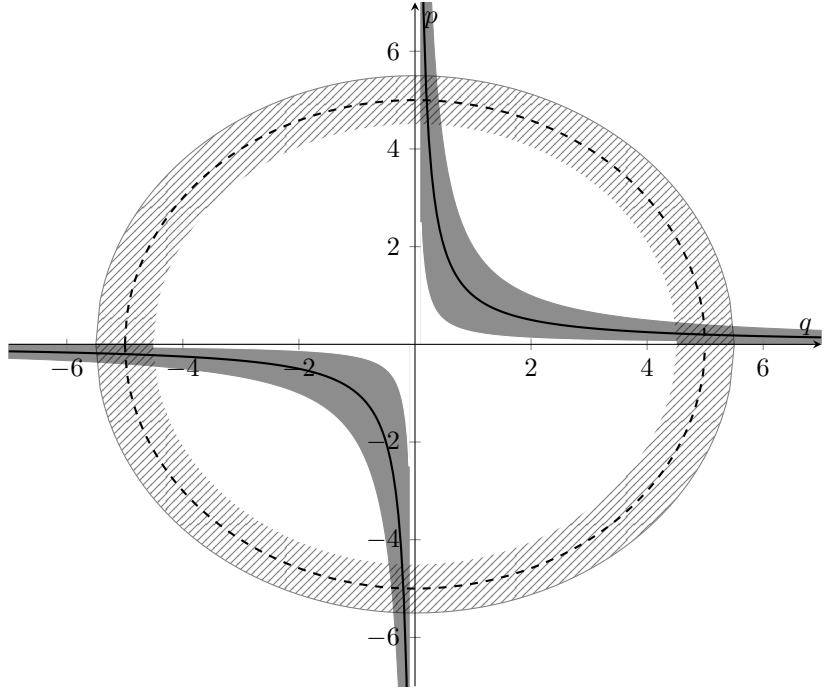
\begin{figure}[t]
  \centering
  \begin{tikzpicture}[scale=1.0, domain=-4:4]
      \begin{axis}[width = 0.75\textwidth, ymin=-7, ymax=7, xmin=-7, xmax=7, xlabel=\(q\), ylabel=\(p\), axis y line = center, axis x line = middle, axis on top] 

        \addplot[gray,fill,opacity=1.0,pattern=north east lines,pattern color = gray][domain=-5.5:5.5,samples=150]{-sqrt(30.25-x*x)}\closedcycle;        
        \addplot[gray,fill,opacity=1.0,pattern=north east lines,pattern color = gray][domain=-5.5:5.5,samples=150]{sqrt(30.25-x*x)}\closedcycle;        
        \addplot[white,fill,opacity=1.0][domain=-4.5:4.5,samples=150]{-sqrt(20.25-x*x)}\closedcycle;        
        \addplot[white,fill,opacity=1.0][domain=-4.5:4.5,samples=150]{sqrt(20.25-x*x)}\closedcycle;        

        \addplot[black,fill,opacity=0.45][domain=-7:-0.1,samples=150]{2.0/x}\closedcycle;
        \addplot[black,fill,opacity=0.45][domain=0.1:7,samples=150]{2.0/x}\closedcycle;
        \addplot[white,fill,opacity=1.0][domain=-4.5:-0.1,samples=150]{0.25/x}\closedcycle;
        \addplot[white,fill,opacity=1.0][domain=0.1:4.5,samples=150]{0.25/x}\closedcycle;

        \addplot[black, thick][domain=-7:-0.1,samples=150]{1/x};
        \addplot[black, thick][domain=0.1:7,samples=150]{1/x};
        \addplot[black, thick, style = dashed][domain=-5:5,samples=300]{sqrt(25-x*x)};
        \addplot[black, thick, style = dashed][domain=-5:5,samples=300]{-sqrt(25-x*x)};

      \end{axis}
  \end{tikzpicture}
  \caption{Conserved and approximately conserved Noether quantities and the resulting constrained solution space. Suppose that both \(p^{T}q = 1\) and \(p^{2} + q^{2} = 5\) were conserved quantities for a certain Lagrangian. Then the solutions of the Euler-Lagrange equations would be constrained to the intersections of these two constant surfaces in phase space; in the above diagram, this is the intersection of the dashed and solid lines. If these quantities were conserved up to a fixed error along a numerical solution, then the numerical solution would be constrained to the intersection of the shaded regions in the above figure. The constraint of the numerical solution to these regions is what leads to the many excellent qualities of variational integrators.}
\end{figure}

One of the great advantages of using variational integrators for problems in geometric mechanics is that by construction they have a rich geometric structure which helps lead to excellent long term and qualitative behavior. An important geometric feature of variational integrators is the preservation of discrete Noether quantities, which are invariants that are derived from symmetries of the action. These are analogous to the more familiar Noether quantities of geometric mechanics in the continuous case. We quickly recall Noether's theorem in both the discrete and continuous case, which will also help define the notation used throughout the proofs that follow. The proofs of both these theorems can be found in \citet{HaLuWa2006}.%
\begin{theorem}\emph{(Noether's Theorem)} Consider a system with Hamiltonian \(H\lefri{p,q}\) and Lagrangian \(L\lefri{q,\dot{q}}\). Suppose \(\left\{g_{s}:s\in\mathbb{R}\right\}\) is a one-parameter group of transformations which leaves the Lagrangian invariant. Let
\begin{align*}
 a\lefri{q} = \left.\frac{d}{d\mbox{s}}\right|_{s=0}g_{s}\lefri{q}
\end{align*}%
be defined as the vector field with flow \(g_{s}\lefri{q}\), referred to as the infinitesimal generator, and define the canonical momentum
\begin{align*}
p = \dLddq\lefri{q,\dot{q}}.
\end{align*}%
Then
\begin{align*}
I\lefri{p,q} = p^{T}a\lefri{q}
\end{align*}%
is a first integral of the Hamiltonian system.
\end{theorem}%
\begin{theorem}\emph{(Discrete Noether's Theorem)} Suppose the one-parameter group of transformations leaves the discrete Lagrangian \(L_{d}\lefri{q_{k},q_{k+1}}\) invariant for all \(\lefri{q_{k},q_{k+1}}\). Then:%
\begin{align*}
p_{k+1}^{T}a\lefri{q_{k+1}} = p_{k}^{T}a\lefri{q_{k}} 
\end{align*}%
where%
\begin{align*}
p_{k} &= -D_{1}L_{d}\lefri{q_{k},q_{k+1}}, \\
p_{k+1} &= D_{2}L_{d}\lefri{q_{k},q_{k+1}}.
\end{align*}
\end{theorem}%
For the remainder of this section, we will refer to \(I\lefri{q,p}\) as the \emph{Noether quantity} and \(p_{n}^{T}a\lefri{q_{n}} = p_{n+1}^{T}a\lefri{q_{n+1}}\) as the \emph{discrete Noether quantity}.

 For Galerkin variational integrators, it is possible to bound the error of the Noether quantities along the Galerkin curve from the behavior of the analogous discrete Noether quantities of the discrete problem and, more importantly, this bound is independent of the number of time steps that are taken in the numerical integration. This is significant because it offers insight into the excellent behavior of spectral variational integrators even over long periods of integration.

The proof of convergence and near preservation of Noether quantities is broken into three major parts. First, we note that on step \(k\) of a numerical integration the discrete Noether quantity arises from a function of the Galerkin curve and the initial point of the one-step map \(\lefri{q_{k-1},q_{k}}\), and that a bound exists for the difference of this discrete Noether quantity evaluated on the Galerkin curve and evaluated on the local exact solution to the Euler-Lagrange equations \(\truq\). Second, we show that a bound exists for the difference of the discrete Noether quantity on the local exact solution of the Euler-Lagrange equations and the value of the Noether quantity of the local exact solution, which is conserved along the flow of the Euler-Lagrange vector field. Finally, we show that under certain smoothness conditions, there exists a point-wise bound between the Noether quantity evaluated on the Galerkin curve and the Noether quantity evaluated on the local exact solution. Thus, we establish a point-wise bound between the Noether quantity evaluated on the Galerkin curve and the discrete Noether quantity, and a bound between the discrete Noether quantity and the Noether quantity, which leads to a point-wise bound between the Noether quantity evaluated on the Galerkin curve, and the Noether quantity which is conserved along the global flow of the Euler-Lagrange vector field.

Throughout this section we will make the simplifying assumptions that%
\begin{align*}
\galqn = \sum_{i=1}^{n} q_{k}^{i}\phi_{i}
\end{align*}%
where \(q_{k}^{1} = q_{k}\), and thus
\begin{align*}
\frac{\partial \galqn}{\partial q_{k}} = \phi_{1}.
\end{align*}%
This assumption significantly simplifies the analysis.

We begin by bounding the discrete Noether quantity by a function of the local exact solution of the Euler-Lagrange equations.%
 
\begin{lemma}\emph{(Bound on Discrete Noether Quantity)} \label{DNBound} Define the Galerkin Noether map as:%
\begin{align*}
I_{d}\lefri{q\lefri{t},q_{k}} &= -\lefri{h\sum_{j=1}^{n}b_{j}\left[\dLdq\lefri{q,\dot{q}}\phi_{1} + \dLddq\lefri{q,\dot{q}}\dot{\phi}_{1}\right]}^{T}a\lefri{q_{k}}
\end{align*}%
and note that the discrete Noether quantity is given by%
\begin{align*}
I_{d}\lefri{\galqn,q_{k}} = p_{n}^{T}a\lefri{q_{k}}.
\end{align*}
Assuming the quadrature accuracy of Theorem (\ref{SpecConv}) with \(n\)-refinement and Theorem (\ref{OptConv}) with \(h\)-refinement, if \(\dLdq\lefri{q,\dot{q}}\), \(\dLddq\lefri{q,\dot{q}}\) and \(\frac{\mbox{d}}{\dt}\dLddq\) are Lipschitz continuous, \(\LNorm{\phi_{1}}{\infty}\) is bounded with \(n\) refinement, and \(\SobNorm{\galqn - \truq}{1}\) is bounded below by the quadrature error, then%
\begin{align*}
\left|I_{d}\lefri{\galqn,q_{k}} - I_{d}\lefri{\truq,q_{k}}\right| \leq C \left|a\lefri{q_{k}}\right|\lefri{\SobNorm{\galqn - \truq}{1} + \LNorm{\galqn - \truq}{\infty} + \LNorm{\dgalqn - \dtruq}{\infty}}
\end{align*}
for some \(C\) independent of \(n\) and \(h\).
\end{lemma}

\begin{proof}
We begin by expanding the definitions of the discrete Noether quantity:
\begin{align*}
\left|I_{d}\lefri{\galqn,q_{k}} - I_{d}\lefri{\truq,q_{k}}\right|=& \left|h\lefri{\sum_{j=1}^{m} b_{j}\left[\dLdq\lefri{\galqn,\dgalqn}\phi_{1} + \dLddq \lefri{\galqn,\dgalqn}\dot{\phi}_{1}\right]}^{T}a\lefri{q_{k}} \right. \\
& \hspace{5em} - \left. \lefri{h\sum_{j=1}^{m} b_{j}\left[\dLdq\lefri{\truq,\dtruq}\phi_{1} + \dLddq \lefri{\truq,\dtruq}\dot{\phi}_{1}\right]}^{T}a\lefri{q_{k}}\right| \\ 
=& \left|\lefri{h\sum b_{j} \left[\lefri{\dLdq\lefri{\galqn,\dgalqn} - \dLdq\lefri{\truq,\dtruq}}\phi_{1} - \lefri{\dLddq\lefri{\galqn,\dgalqn} - \dLddq\lefri{\truq,\dtruq}}\dot{\phi_{1}}\right]}^{T}a\lefri{q_{k}}\right| \\
\leq & \left|h\sum_{j=1}^{m} b_{j} \left[\lefri{\dLdq\lefri{\galqn,\dgalqn} - \dLdq\lefri{\truq,\dtruq}}\phi_{1} - \lefri{\dLddq\lefri{\galqn,\dgalqn} - \dLddq\lefri{\truq,\dtruq}}\dot{\phi_{1}}\right] \right|\left|a\lefri{q_{k}}\right|.
\end{align*}%
Now we introduce the function \(\quaderr{\cdot}{\cdot}\) which gives the error of the quadrature rule, and thus
\begin{align*}
\left|I_{d}\lefri{\galqn,q_{k}} - I_{d}\lefri{\truq,q_{k}}\right| \leq & \left|\int_{0}^{h}\lefri{\dLdq\lefri{\galqn,\dgalqn} - \dLdq\lefri{\truq,\dtruq}}\phi_{1} - \lefri{\dLddq\lefri{\galqn,\dgalqn} - \dLddq\lefri{\truq,\dtruq}}\dot{\phi}_{1} \dt\right. \\
& \hspace{5em} \left. \vphantom{\int_{0}^{h}\dLddq} + \quaderr{\galqn - \truq}{\dgalqn - \dtruq}\right| \left|a\lefri{q_{k}}\right|. 
\end{align*}%
Integrating by parts, we get:%
\begin{align*}
\left|I_{d}\lefri{\galqn,q_{k}} - I_{d}\lefri{\truq,q_{k}}\right| \leq & \left|\int_{0}^{h}\lefri{\dLdq\lefri{\galqn,\dgalqn} - \dLdq\lefri{\truq,\dtruq}}\phi_{1} - \frac{\mbox{d}}{\dt}\lefri{\dLddq\lefri{\galqn,\dgalqn} - \dLddq\lefri{\truq,\dtruq}}\phi_{1} \dt\right. \\
& \hspace{5em} \left. \vphantom{\int_{0}^{h}\dLddq} + \left.\lefri{\dLddq\lefri{\galqn,\dgalqn} - \dLddq\lefri{\truq,\dtruq}}\phi_{1}\right|_{0}^{h} + \quaderr{\galqn - \truq}{\dgalqn - \dtruq}\right| \left|a\lefri{q_{k}}\right|. 
\end{align*}%
Introducing the Lipschitz constants \(L_{1}\) for \(\dLdq\), \(L_{2}\) for \(\dLddq\), and \(L_{3}\) for \(\frac{d}{dt}\dLddq\),%
\begin{align*}
\left|I_{d}\lefri{\galqn,q_{k}} - I_{d}\lefri{\truq,q_{k}}\right| \leq& \left(\int_{0}^{h} \lefri{L_{1} + L_{3}}\left|\lefri{\galqn,\dgalqn} - \lefri{\truq,\dtruq}\right|\left|\phi_{1}\right| \dt + 2L_{2}\lefri{\LNorm{\phi_{1}}{\infty}} \right. \\
& \hspace{5em}  \left. \vphantom{\int_{0}^{h}} \left(\LNorm{\galqn - \truq}{\infty} + \LNorm{\dgalqn - \dtruq}{\infty}\right) + \quaderr{\galqn - \truq}{\dgalqn - \dtruq} \right) \left|a\lefri{q_{k}}\right|\\
& \leq \lefri{L_{1} + L_{3}}\LNorm{\phi_{1}}{\infty}\left|a\lefri{q_{k}}\right|\lefri{\int_{0}^{h} \left|\lefri{\galqn,\dgalqn} - \lefri{\truq,\dtruq}\right|\dt} \\
& \hspace{5em} + 2L_{2}\lefri{\LNorm{\phi_{1}}{\infty}}\left|a\lefri{q_{k}}\right|\lefri{\LNorm{\galqn - \truq}{\infty} + \LNorm{\dgalqn - \dtruq}{\infty}} \\
& \hspace{5em} + \quaderr{\galqn - \truq}{\dgalqn - \dtruq}\left|a\lefri{q_{k}}\right|.
\end{align*}%
We now make the simplification that the quadrature error \(\left|\quaderr{\cdot}{\cdot}\right|\) serves as a lower bound for \(\SobNorm{\galqn - \truq}{1}\). While this may not strictly hold, all of our estimates on the convergence for \(\galqn\) imply this bound, and hence it is a reasonable simplification for establishing convergence in this case. Now, note that \(\LNorm{\phi_{1}}{\infty}\) is invariant under \(h\) rescaling, and let%
\begin{align*}
C = \max\lefri{L_{1} + L_{3},2L_{2}}\LNorm{\phi_{1}}{\infty} + 1
\end{align*}%
to get%
\begin{align*}
\left|I_{d}\lefri{\galqn,q_{k}} - I_{d}\lefri{\truq,q_{k}}\right| \leq C\left|a\lefri{q_{k}}\right|\lefri{\SobNorm{\galqn - \truq}{1} + \LNorm{\galqn - \truq}{\infty} + \LNorm{\dgalqn-\dtruq}{\infty}}
\end{align*}%
which establishes the result.
\end{proof}

Lemma \ref{DNBound} establishes a bound between the discrete Noether quantity and \(I_{d}\lefri{\truq,q_{k}}\). The next step is to establish a bound between \(I_{d}\lefri{\truq,q_{k}}\) and the Noether quantity. 

\begin{lemma}\label{DNandTNBound}\emph{(Error Between Discrete Noether Quantity and True Noether Quantity)}  Assume that \(\phi_{1}\lefri{0} = 1\) and \(\phi_{1}\lefri{h} = 0 \), and that the sequence \(\left\{\left|a\lefri{q_{k}}\right|\right\}_{k=1}^{N}\) is bounded by a constant \(\akC\) which is independent of \(N\). Let%
\begin{align*}
\trup\lefri{t} = \dLddq\lefri{\truq\lefri{t},\dtruq\lefri{t}}. 
\end{align*}
Once again, let the error of the quadrature rule be given by \(\quaderr{\cdot}{\cdot}\). Then
\begin{align*}
\left|I^{d}\lefri{\truq,q_{k}} - I\lefri{\trup\lefri{t},\truq\lefri{t}}\right| &\leq \akC \left|\quaderr{\truq}{\dtruq}\right| 
\end{align*}%
for any \(t \in \left[0,h\right]\).
\end{lemma}

\begin{proof}
First, we note that since \(\truq\) solves the Euler-Lagrange equations exactly, \(I\lefri{\trup\lefri{t},\truq\lefri{t}}\) is a conserved quantity along the flow, so it suffices to show the inequality holds for \(t=0\). We begin by expanding:
\begin{align*}
\left|I^{d}\lefri{\truq,q_{k}} - I\lefri{\trup\lefri{0},\truq\lefri{0}}\right| &= \left|-h\lefri{\sum_{j=1}^{m}b_{j}\dLdq\lefri{\truq,\dtruq}\phi_{1} + \dLddq\lefri{\truq,\dtruq}\dot{\phi}_{1}}^{T}a\lefri{q_{k}} - \trup\lefri{0}^{T}a\lefri{\truq\lefri{0}}\right| \\
&= \left|-\lefri{\int_{0}^{h}\dLdq\lefri{\truq,\dtruq}\phi_{1} + \dLddq\lefri{\truq,\dtruq}\dot{\phi}_{1}\dt + \quaderr{\truq}{\dtruq}}^{T}a\lefri{q_{k}} - \trup\lefri{0}^{T}a\lefri{\truq\lefri{0}} \right|\\
&= \left|-\left(\int_{0}^{h}\lefri{\dLdq\lefri{\truq,\dtruq} - \frac{d}{d\mbox{t}}\dLddq\lefri{\truq,\dtruq}}\phi_{1}\dt + \dLddq\lefri{\truq\lefri{h},\dtruq\lefri{h}}\phi_{1}\lefri{h} \right. \right. \\
& \left. \left. \hspace{5em} - \vphantom{\int_{0}^{h}} \dLddq\lefri{\truq\lefri{0},\dtruq\lefri{0}}\phi_{1}\lefri{0} + \quaderr{\truq}{\dtruq}\right)^{T} a\lefri{q_{k}} - \trup\lefri{0}^{T}a\lefri{\truq\lefri{0}} \right|
\end{align*}%
Since \(\truq\lefri{t}\) solves the Euler-Lagrange equations, \(\phi_{1}\lefri{0} = 1\) and \(\phi_{1}\lefri{h} = 0\), and \(\truq\lefri{0} = q_{k}\),
\begin{align*}
\left|I^{d}\lefri{\truq,q_{k}} - I\lefri{\trup\lefri{0},\truq\lefri{0}}\right| =& \left|\lefri{\dLddq\lefri{\truq\lefri{0},\dtruq\lefri{0}}}^{T}a\lefri{q_{k}} + \lefri{\quaderr{\truq}{\dtruq}}^{T}a\lefri{q_{k}} - \trup\lefri{0}^{T}a\lefri{q_{k}}\right|\\
=& \left|\lefri{\trup\lefri{0}}^{T}a\lefri{q_{k}} + \lefri{\quaderr{\truq}{\dtruq}}^{T}a\lefri{q_{k}} - \lefri{\trup\lefri{0}}^{T}a\lefri{q_{k}}\right|\\
=& \left|\quaderr{\truq}{\dtruq}^{T}a\lefri{q_{k}}\right|\\
\leq & \left|\quaderr{\truq}{\dtruq}\right| \left|a\lefri{q_{k}}\right|\\
\leq & \akC \left|\quaderr{\truq}{\dtruq}\right|
\end{align*}%
which yields the desired bound.
\end{proof}

Once again, if we assume that the quadrature error serves as a lower bound for the Sobolev error, combining the bounds from (\ref{DNBound}) and (\ref{DNandTNBound}) yields:%
\begin{align*}
\left|I_{d}\lefri{\galqn,q_{k}} - I\lefri{\trup\lefri{t},\truq\lefri{t}}\right| & \leq 2C\akC\lefri{\SobNorm{\galqn - \truq}{1} + \LNorm{\galqn - \truq}{\infty} + \LNorm{\dgalqn - \dtruq}{\infty}}.
\end{align*}%
This bound serves two purposes; the first is to establish a bound between the discrete Noether quantity and the Noether quantity computed on the local exact solution \(\truq\). The second is to establish a bound between the discrete Noether quantity after one step and the Noether quantity computed on the initial data:%
\begin{align*}
\left|I_{d}\lefri{\galqn,q_{1}} - I\lefri{p\lefri{0},q\lefri{0}}\right| &\leq \
2C\akC\lefri{\SobNorm{\galqn - \truq}{1} + \LNorm{\galqn - \truq}{\infty} + \LNorm{\dgalqn - \dtruq}{\infty}},
\end{align*}%
since for \(\lefri{q_{1},q_{2}}\), \(\truq\) is the \emph{global} exact flow of the Euler-Lagrange equations.

The difference between these two bounds is subtle but important; by establishing a bound between the discrete Noether quantity and the Noether quantity associated with the initial conditions, on any step of the method we can bound the error between the discrete Noether quantity and the Noether quantity associated with the \emph{global} exact flow. By establishing the bound between the discrete Noether quantity and the Noether quantity associated with \(\truq\) at any step, we can bound the error between the Noether quantity associated with the local exact flow \(\truq\) and the true Noether quantity conserved along the global exact flow:%
\begin{align}
\left|I\lefri{\trup\lefri{t},\truq\lefri{t}} - I\lefri{p\lefri{0},q\lefri{0}}\right| &\leq \left|I\lefri{\trup\lefri{t},\truq\lefri{t}} - I_{d}\lefri{\galqn,q_{k}}\right| + \left|I_{d}\lefri{\galqn,q_{k}} - I\lefri{p\lefri{0},q\lefri{0}}\right| \nonumber \\
&\leq 4C\akC\lefri{\SobNorm{\galqn - \truq}{1} + \LNorm{\galqn - \truq}{\infty} + \LNorm{\dgalqn - \dtruq}{\infty}} \label{SuperImportantBound}
\end{align}%
for any \(t_{0} \in \left[0,h\right]\) on any time step \(k\). Because the local exact flow \(\truq\) is generated from boundary conditions \(\lefri{q_{k},q_{k+1}}\) which only approximate the boundary conditions of the true flow, there is no guarantee that the Noether quantity associated with \(\truq\) will be the same step to step, only that it will be conserved within each time step. However, because there is a bound between the Noether quantity associated with \(\truq\) and the discrete Noether quantity at every time step, the discrete Noether quantity and the Noether quantity associated with the exact flow, and because the Noether quantity is conserved point-wise along \(\truq\) on each time step, there exists a bound between the Noether quantity associated with each point of the local exact flow and the Noether quantity associated with the true solution.

We finally arrive at the desired result, which is a theorem that bounds the error between the Noether quantity along the Galerkin curve and the true Noether quantity. It is significant because not only does it bound the error of the Noether quantity, but the bound is independent of the number of steps taken, and hence will not grow even for extremely long numerical integrations.

\begin{theorem}\emph{(Convergence of Conserved Noether Quantities)}\label{ConvNQ} Define%
\begin{align*}
\galpn = \dLddq\lefri{\galqn,\dgalqn}.
\end{align*}
 Under the assumptions of Lemmas (\ref{DNBound} - \ref{DNandTNBound}), if the Noether map \(I\lefri{p,q}\) is Lipschitz continuous in both its arguments, then there exists a constant \(\NoeC\) independent \(N\), the number of method steps, such that:%
\begin{align*}
\left|I\lefri{p\lefri{0},q\lefri{0}} - I\lefri{\tilde{p}_{n}\lefri{t},\galqn\lefri{t}}\right| \leq \NoeC\lefri{\SobNorm{\galqn - \truq}{1} + \LNorm{\galqn - \truq}{\infty} + \LNorm{\dgalqn - \dtruq}{\infty}}.
\end{align*}%
for any \(t \in \left[0,Nh\right]\).
\end{theorem}%
%
%
\begin{proof}We begin by introducing the Noether quantity evaluated at \(t\) on the local exact flow, \(\truq\):
\begin{align}
\left|I\lefri{p\lefri{0},q\lefri{0}}-I\lefri{\galpn\lefri{t},\galqn\lefri{t}}\right| \leq& \left|I\lefri{\galpn\lefri{t},\galqn\lefri{t}} - I\lefri{\trup\lefri{t},\truq\lefri{t}}\right| \label{NoeTwoTerm}\\
& \hspace{5em} + \left|I\lefri{\trup\lefri{t},\truq\lefri{t}} - I\lefri{p\lefri{0},q\lefri{0}}\right|. \nonumber
\end{align}%
Considering the first term in (\ref{NoeTwoTerm}), let \(L_{4}\) be the Lipschitz constant for \(I\lefri{\cdot,\cdot}\). Then%
\begin{align}
\left|I\lefri{\galpn\lefri{t},\galqn\lefri{t}} - I\lefri{\trup\lefri{t},\truq\lefri{t}}\right| &\leq L_{4}\left|\lefri{\galpn\lefri{t},\galqn\lefri{t}} - \lefri{\trup\lefri{t},\truq\lefri{t}}\right| \nonumber \\
& \leq L_{4} \lefri{\left|\galpn\lefri{t} - \trup\lefri{t}\right| + \left|\galqn\lefri{t} - \trup\lefri{t}\right|} \nonumber \\
& = L_{4} \lefri{\left|\dLddq\lefri{\galqn\lefri{t},\dgalqn\lefri{t}} - \dLddq\lefri{\truq\lefri{t},\dtruq{\lefri{t}}}\right| + \left|\galqn\lefri{t} - \truq\lefri{t}\right|} \nonumber \\
& \leq L_{4}\lefri{L_{2}\left|\dgalqn\lefri{t} - \dtruq\lefri{t}\right| + \lefri{L_{2}+1}\left|\galqn\lefri{t} - \truq\lefri{t}\right|} \nonumber\\
& \leq L_{4}\lefri{L_{2} + 1}\lefri{\LNorm{\galqn - \truq}{\infty} + \LNorm{\dgalqn - \dtruq}{\infty}}. \label{LipBound}
\end{align}

The second term in (\ref{NoeTwoTerm}) is exactly the bound given by (\ref{SuperImportantBound}) and thus combining (\ref{LipBound}) and (\ref{SuperImportantBound}) in (\ref{NoeTwoTerm}) and defining \(\NoeC = 4C\akC + L\lefri{L_{2}+1}\), we have:%
\begin{align*}
\left|I\lefri{\galpn\lefri{t},\galqn\lefri{t}} - I\lefri{\trup\lefri{t},\truq\lefri{t}}\right| \leq \NoeC\lefri{\SobNorm{\galqn - \truq}{1} + \LNorm{\galqn - \truq}{\infty} + \LNorm{\dgalqn - \dtruq}{\infty}}
\end{align*}%
which completes the result.
\end{proof}
The convergence and bounds of the Noether quantity evaluated on the Galerkin curve to that of the true solution is hampered by one issue. While Theorems (\ref{GeoGalerk}) and (\ref{OptGalerk}) provide estimates for convergence in the Sobolev norm \(\SobNorm{\cdot}{1}\), Theorem (\ref{ConvNQ}) requires estimates in the \(L^{\infty}\) norm. We can establish a bound for \(\LNorm{\galqn\lefri{t} - \truq\lefri{t}}{\infty}\), but it is much more difficult to establish a general estimate for \(\LNorm{\dgalqn\lefri{t} - \dtruq\lefri{t}}{\infty}\).
%
%
\begin{lemma} \emph{(Bound on \(L^{\infty}\) Norm from Sobolev Norm)}\label{LinSobNorm} For any \(t \in \left[0,h\right]\), the following bound holds:%
\begin{align*}
\left|q\lefri{t}\right| \leq \max\lefri{\frac{1}{h},1} \SobNorm{q}{1}
\end{align*}%
and thus%
\begin{align*}
\LNorm{q}{\infty} \leq \max\lefri{\frac{1}{h},1} \SobNorm{q}{1}.
\end{align*}
\end{lemma}%
%
%
\begin{proof} This is a basic extension of the arguments from Lemma A.1. in \citet{LaTh2003}, generalizing the lemma from the interval \(\left[0,1\right]\) to an interval of arbitrary length, \(\left[0,h\right]\). We note that for any \(t, s \in \left[0,h\right]\), \(q\lefri{t} = q\lefri{s} + \int_{s}^{t} \dot{q}\lefri{u}\mbox{d}u\). Thus:%
\begin{align*}
\left|q\lefri{t}\right| \leq& \left|q\lefri{s}\right| + \int_{0}^{h} \left|\dot{q}\lefri{u}\right| \mbox{d}u \\
\leq& \left|q\lefri{s}\right| + \LNorm{\dot{q}}{1}. 
\end{align*}%
Now, we integrate with respect to \(s\):%
\begin{align*}
\int_{0}^{h} \left|q\lefri{t}\right| \mbox{d}s &\leq \int_{0}^{h}\left|q\lefri{s}\right| \mbox{d}s + \int_{0}^{h}\LNorm{\dot{q}}{1}\mbox{d}s \\
h\left|q\lefri{t}\right| &\leq \lefri{\LNorm{q}{1} + h\LNorm{\dot{q}}{1}}.
\end{align*}
which yields the desired result.
\end{proof}
Under certain assumptions about the behavior of \(\dgalqn - \dtruq\), it is possible to establish bounds on the point-wise error of \(\dgalqn\) from the Sobolev error \(\SobNorm{\galqn - \truq}{1}\). For example, if the length of time that the error is within a given fraction of the max error is proportional to the length of the interval \(\left[0,h\right]\), i.e. there exists \(C_{1},C_{2}\) independent of \(h\): i.e.,%
\begin{align*}
m\lefri{\left\{t\left|\left\|\lefri{\dgalqn\lefri{t} - \dtruq\lefri{t}}\right\| \geq C_{1}\left\|\dgalqn\lefri{t} - \dtruq\lefri{t}\right\|_{\infty}\right.\right\}} \geq C_{2}h,
\end{align*}%
where \(m\) is the Lebesque measure, then it can easily be seen that:%
\begin{align*}
\SobNorm{\galqn - \truq}{1} \geq \int_{0}^{h} \left\|\dgalqn\lefri{t} - \dtruq\lefri{t}\right\| \dt \geq C_{1}C_{2}h\LNorm{\dgalqn - \dtruq}{\infty}. 
\end{align*}%
While we will not establish here that the \(\dgalqn\) converges in the \(L^{\infty}\) norm with the same rate that the Galerkin curve converges in the Sobolev norm, our numerical experiments will show that the Noether quantities tend to converge at the same rate as the Galerkin curve.

\section{Numerical Experiments}

To support the results in this paper, as well as to investigate the efficiency and stability of spectral variational integrators, several numerical experiments were conducted by applying spectral variational techniques to well-known variational problems. For each problem, the spectral variational integrator was constructed using a Lagrange interpolation polynomials at \(n\) Chebyshev points with a Gauss quadrature rule at \(2n\) points. Convergence of both the one-step map and the Galerkin curves was measured using the \(\ell^{\infty}\) and \(L^{\infty}\) norms respectively, although we record them on the same axis using labeled \(L^{\infty}\) error in a slight abuse of notation. The experiments strongly support the results of this paper, and suggest topics for further investigation.

\afterpage{\clearpage}

\begin{figure}[p]
   \centering
   \includegraphics[width=.75\textwidth]{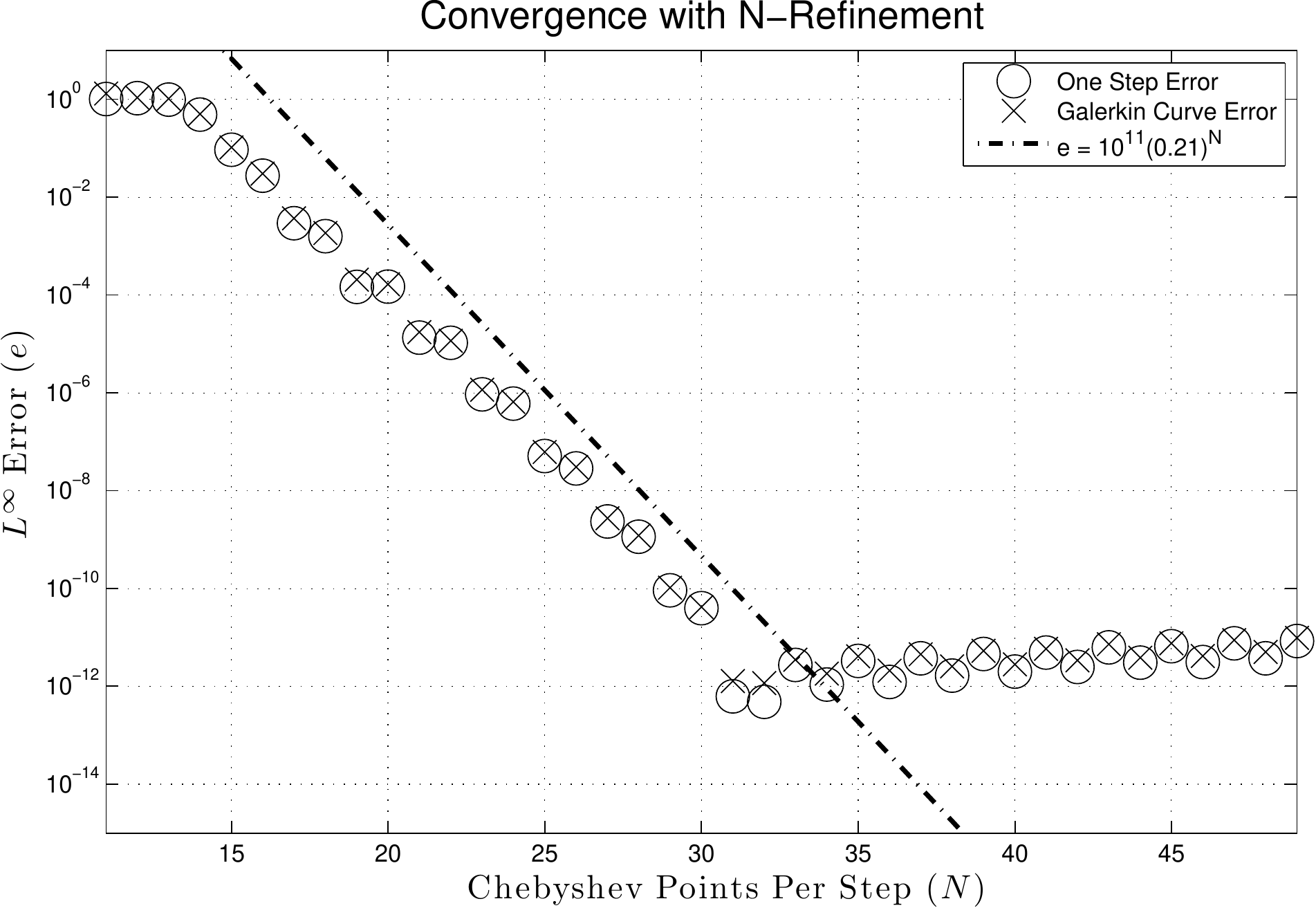}
   \caption{Geometric convergence of the spectral variational integration of the harmonic oscillator problem, for 100 steps at step size \(h=20.0\).}
   \label{fig:HarmOscNRefine}
\end{figure}

\begin{figure}[p]
   \centering
   \includegraphics[width=0.75\textwidth]{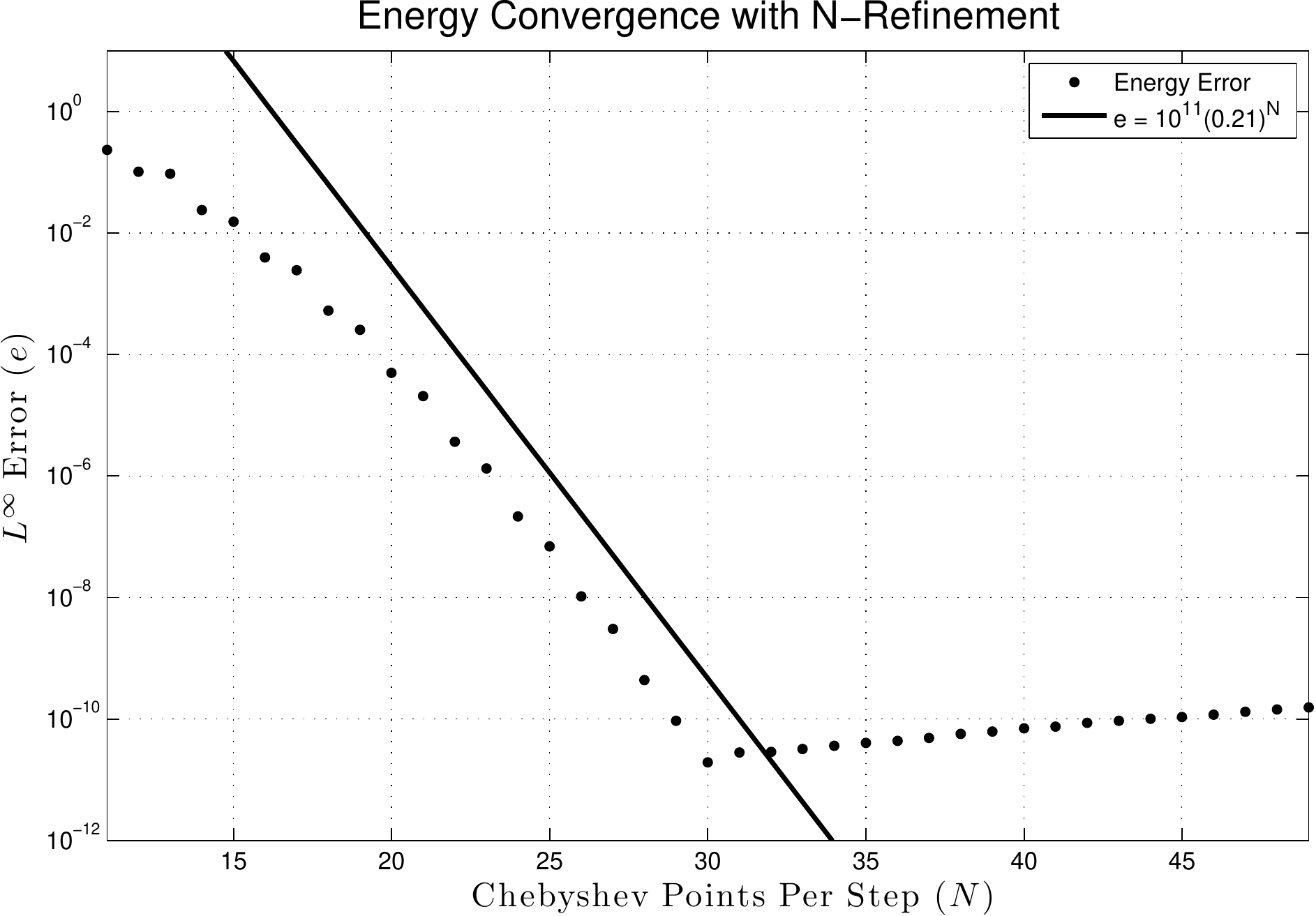}
   \caption{Geometric convergence of the energy error of the spectral variational integration of the harmonic oscillator problem for 100 steps at step size \(h=20.0\).}
   \label{fig:HarmOscEnergyNRefine}
\end{figure}

\subsection{Harmonic Oscillator}%

\begin{figure}[p]
 \centering
 \includegraphics[width=0.75\textwidth]{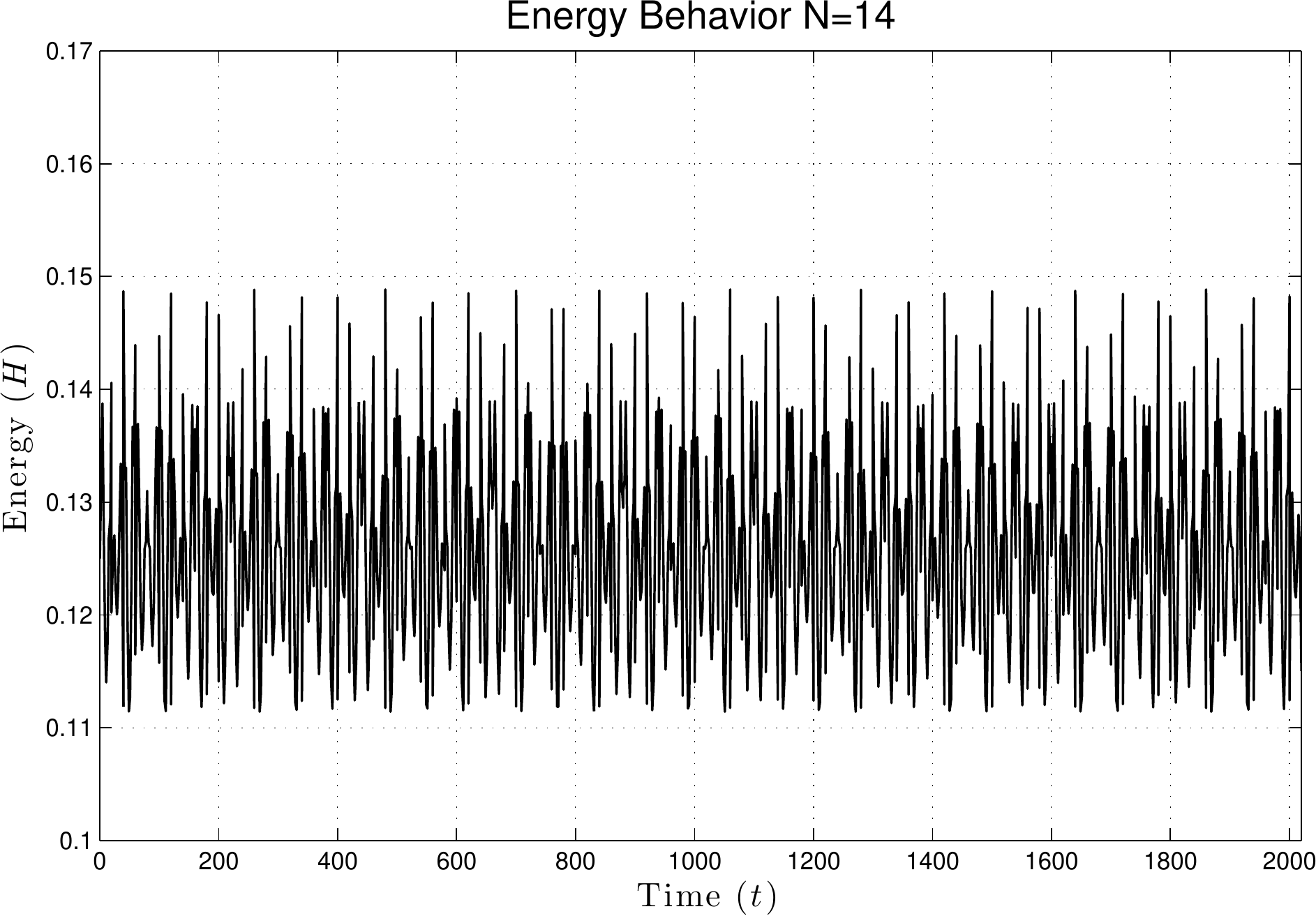}
 \caption{Energy Stability of the spectral variational integration of the harmonics oscillator problem. This energy was computed for the integration using \(n=14\) for step size \(h=20.0\).}
 \label{fig:HarmOscEnergyStability}
\end{figure}
The first and simplest numerical experiment conducted was the harmonic oscillator. Starting from the Lagrangian,%
\begin{align*}
L\lefri{q,\dot{q}} = \frac{1}{2}\dot{q}^{2} - \frac{1}{2}q^{2},
\end{align*}%
where \(q\in \mathbb{R}\), the induced spectral variational discrete Euler-Lagrange equations are linear. Choosing the large time step \(h=20\) over 100 steps yields the expected geometric convergence, and attains very high accuracy, as can be seen in Figure \ref{fig:HarmOscNRefine}. In addition, the max error of the energy converges geometrically, see Figure \ref{fig:HarmOscEnergyNRefine}, and does not grow over the time of integration, see Figure \ref{fig:HarmOscEnergyStability}.

\subsection{N-body Problems}

We now turn our attention towards Kepler \(N\)-body problems, which are both good benchmark problems and are interesting in their own right. The general form of the Lagrangian for these problems is%
\begin{align*}
L\lefri{q,\dot{q}} &= \frac{1}{2} \sum_{i=1}^{N} \dot{q}_{i}^{T}M\dot{q}_{i} + G\sum_{i=1}^{N} \sum_{j=0}^{i-1} \frac{m_{i}m_{j}}{\left\|q_{i}-q_{j}\right\|},
\end{align*}%
where \(q_{i} \in \mathbb{R}^{D}\) is the center of mass for body \(i\), \(G\) is a gravitational constant, and \(m_{i}\) is a mass constant associated with the body described by \(q_{i}\).

\subsubsection{2-Body Problem}

\begin{figure}[p]
  \centering
  \includegraphics[width = 0.75\textwidth]{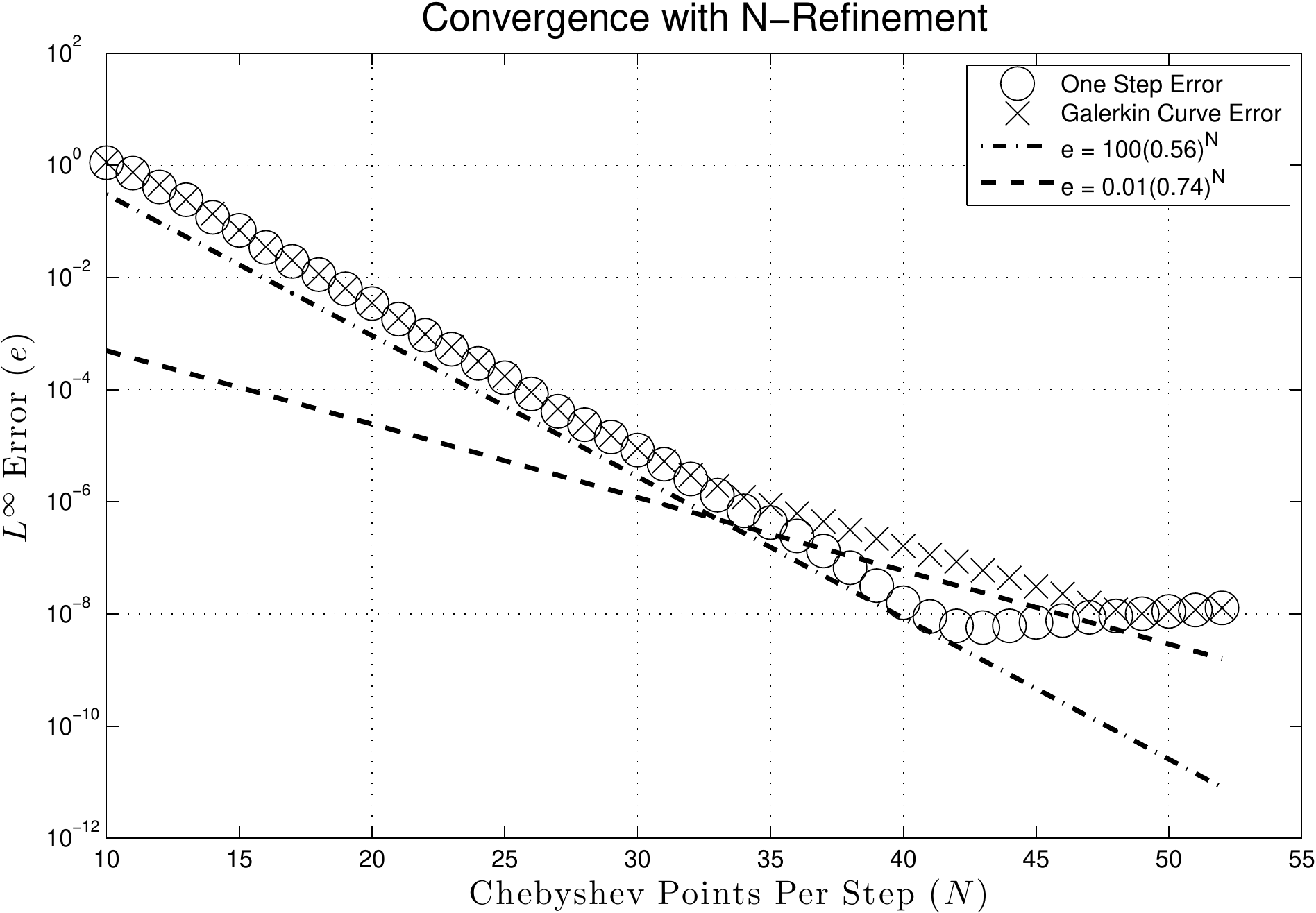}
  \caption{Geometric convergence of the Kepler 2-body problem with eccentricity 0.6 over 100 steps of \(h = 2.0\). Note that around \(n=32\), the error for the Galerkin curves becomes \(\mathcal{O}\lefri{0.74^{n}}\), while the error for the one-step map is always \(\mathcal{O}\lefri{0.56^{n}}\).}
  \label{fig:Kepler2bodyNRefinement}
\end{figure}

\begin{figure}[p]
  \centering
  \includegraphics[width = 0.75\textwidth]{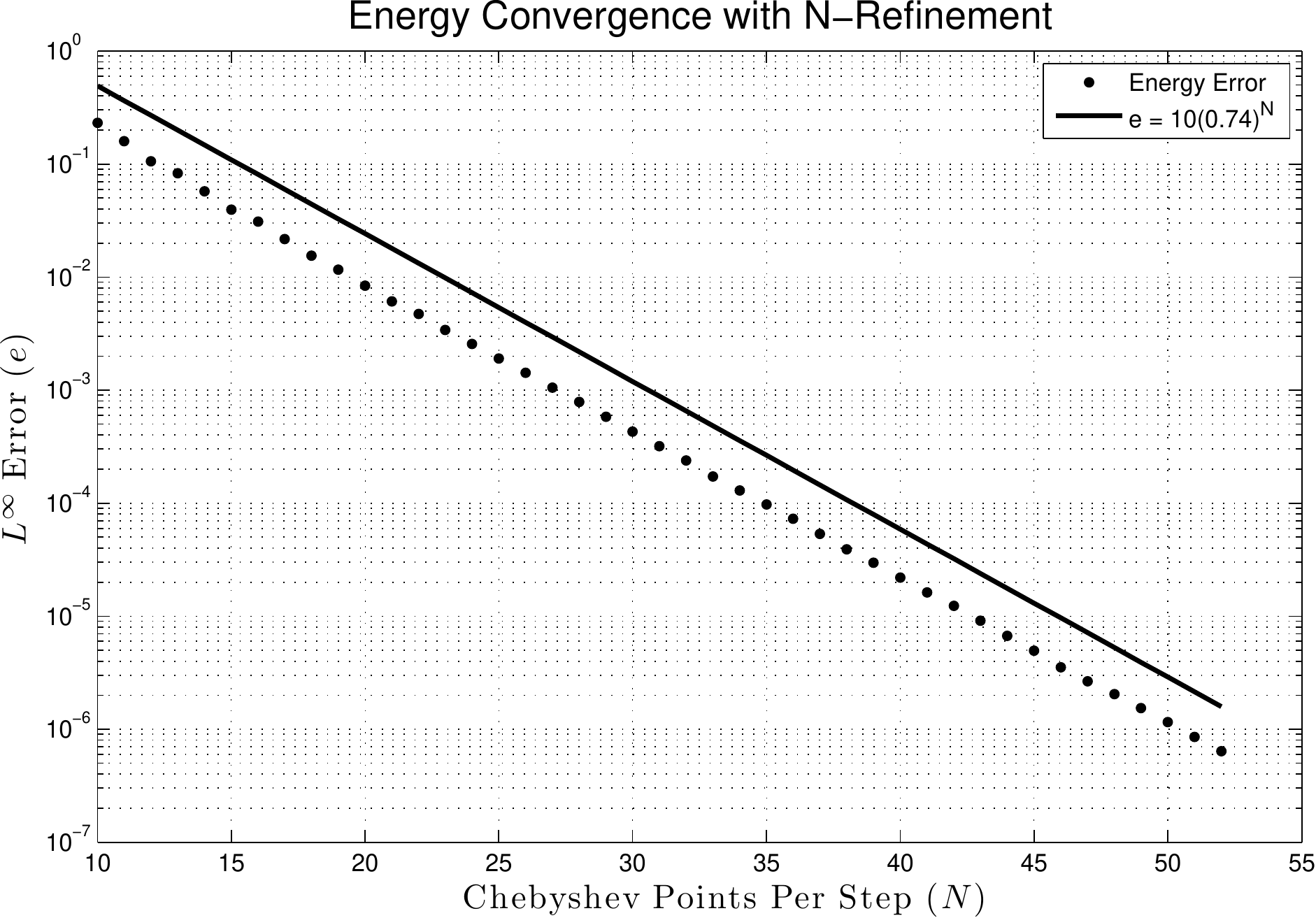}
  \caption{Geometric convergence of the Energy Error of Kepler 2-body problem with eccentricity 0.6 over 100 steps of \(h = 2.0\). Note that the error is \(\mathcal{O}\lefri{0.74^{n}}\), the same as it was for the Galerkin curves.}
  \label{fig:Kepler2bodyNRefinement_Energy}
\end{figure}

\begin{figure}[p]
  \centering
  \includegraphics[width = 0.75\textwidth]{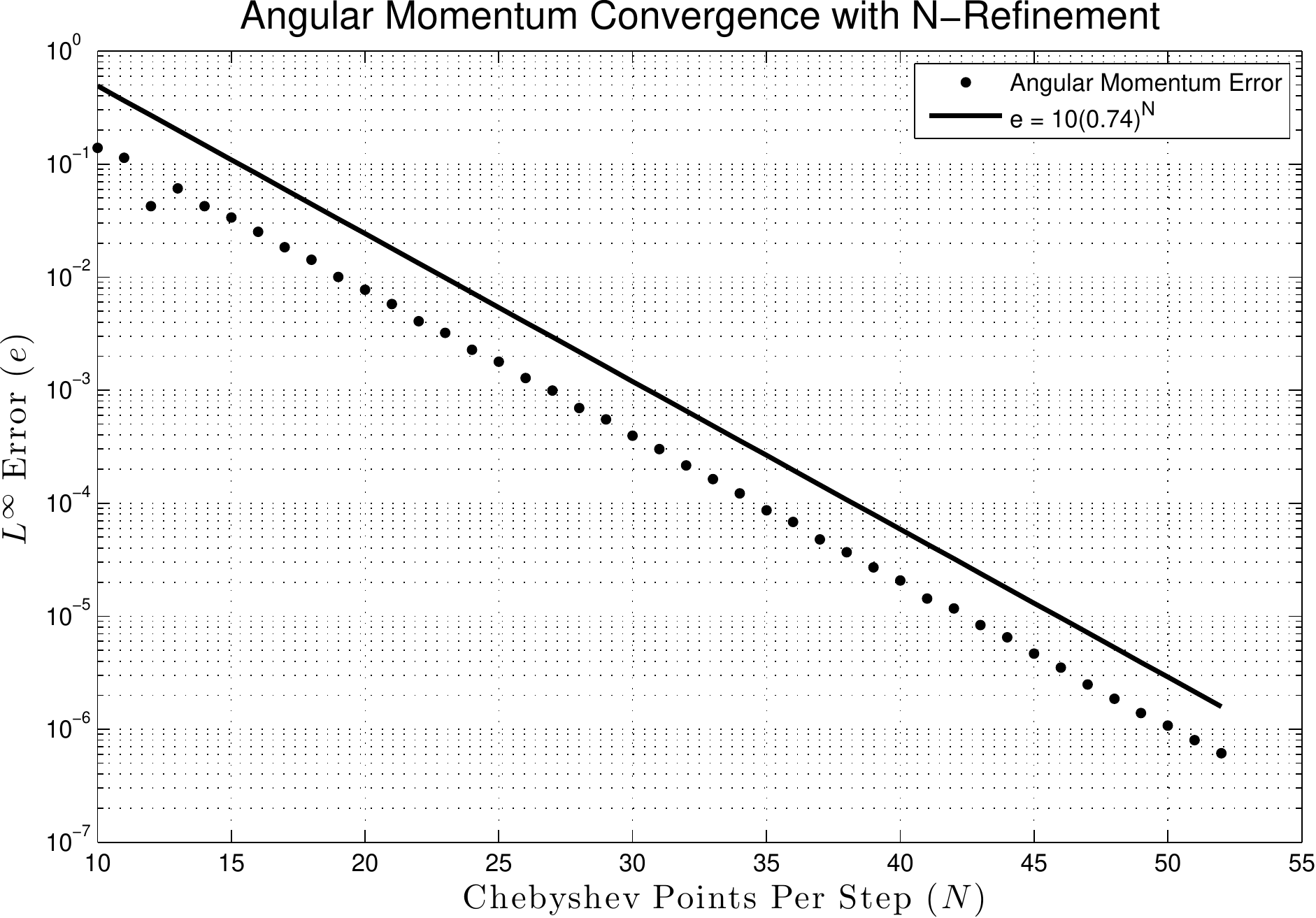}
  \caption{Geometric convergence of the angular momentum of the Kepler 2-body problem with eccentricity 0.6 over 100 steps of \(h = 2.0\). Again, the error is of the same order as it was the Galerkin curves.}
  \label{fig:Kepler2bodyNRefinement_Angular}
\end{figure}

\begin{figure}[p]
  \centering
  \includegraphics[width = 0.75\textwidth]{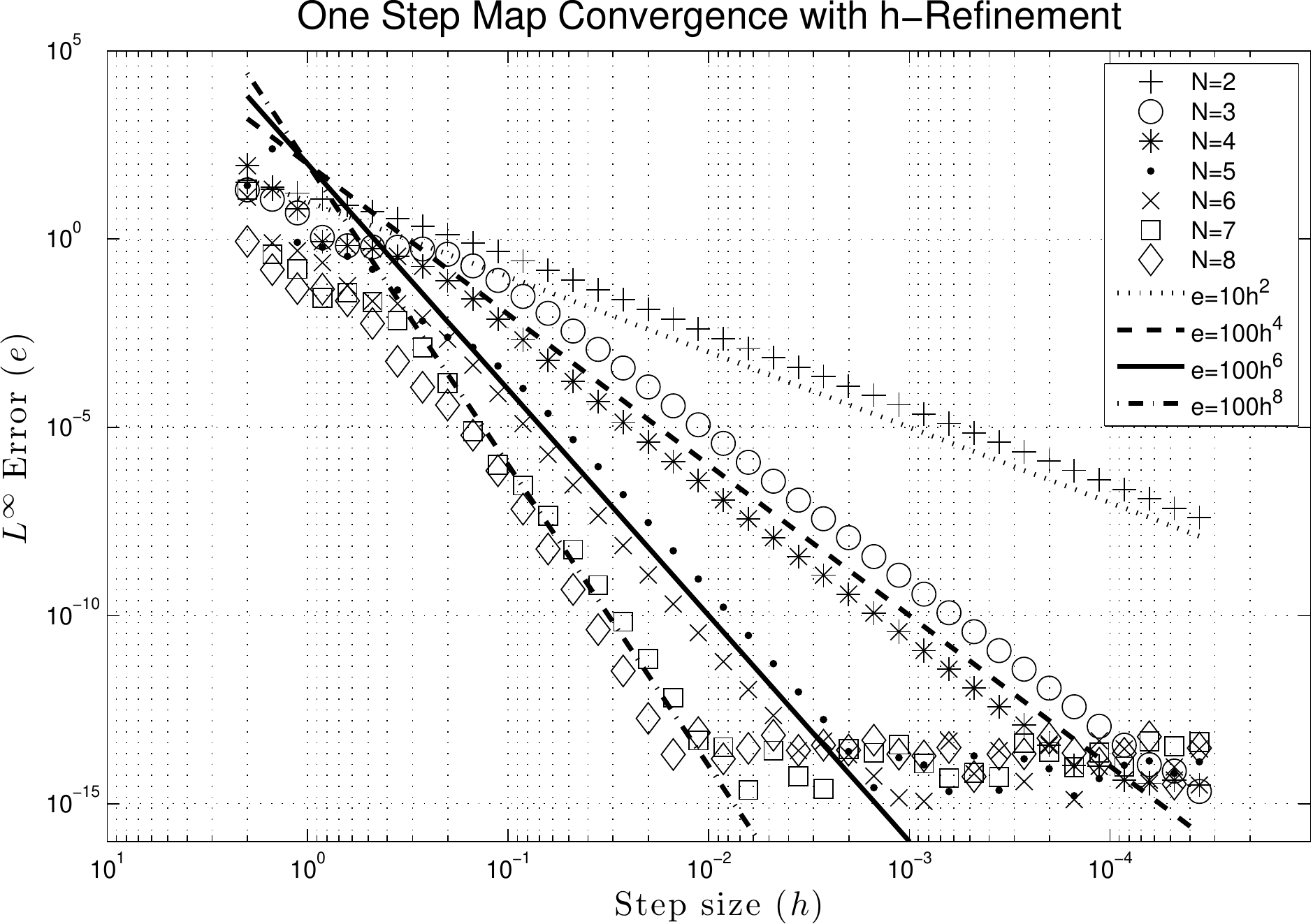}
  \caption{Order Optimal convergence of the Kepler 2-body problem with eccentricity 0.6 over 10 steps with \(h\) refinement. Note our bound is not sharp, as the error is \(\mathcal{O}\lefri{h^{2\left\lceil \frac{n}{2} \right\rceil}}\), where \(\left\lceil \cdot\right\rceil\) is the ceiling function.}
  \label{fig:Kepler2bodyhRefinement}
\end{figure}

\begin{figure}[p]
  \centering
  \includegraphics[width = 0.75\textwidth]{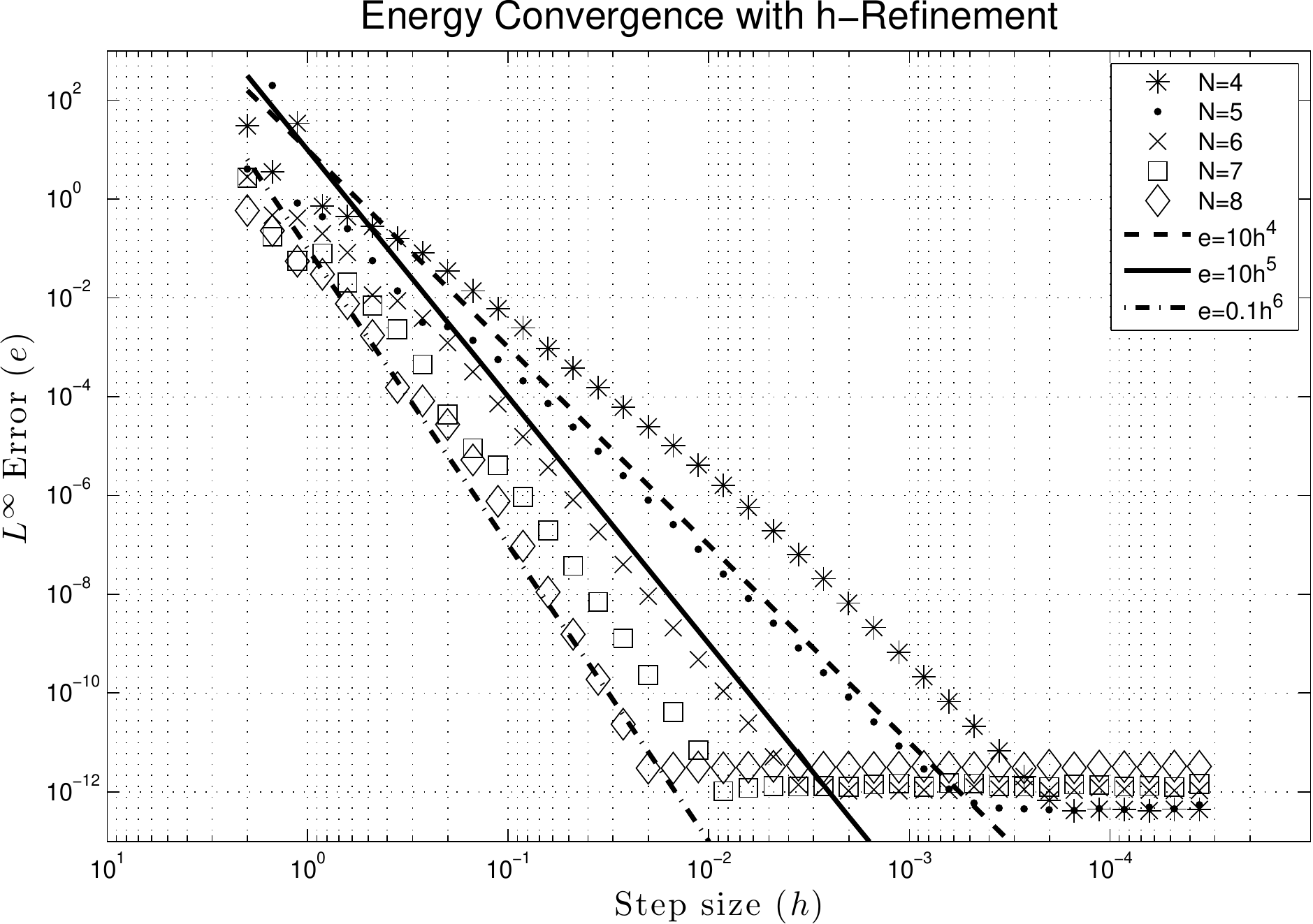}
  \caption{Convergence of the Kepler 2-body problem energy with eccentricity 0.6 over 10 steps with \(h\) refinement.}
  \label{fig:Kepler2bodyhRefinement_energy}
\end{figure}

\begin{figure}[p]
  \centering
  \includegraphics[width = 0.75\textwidth]{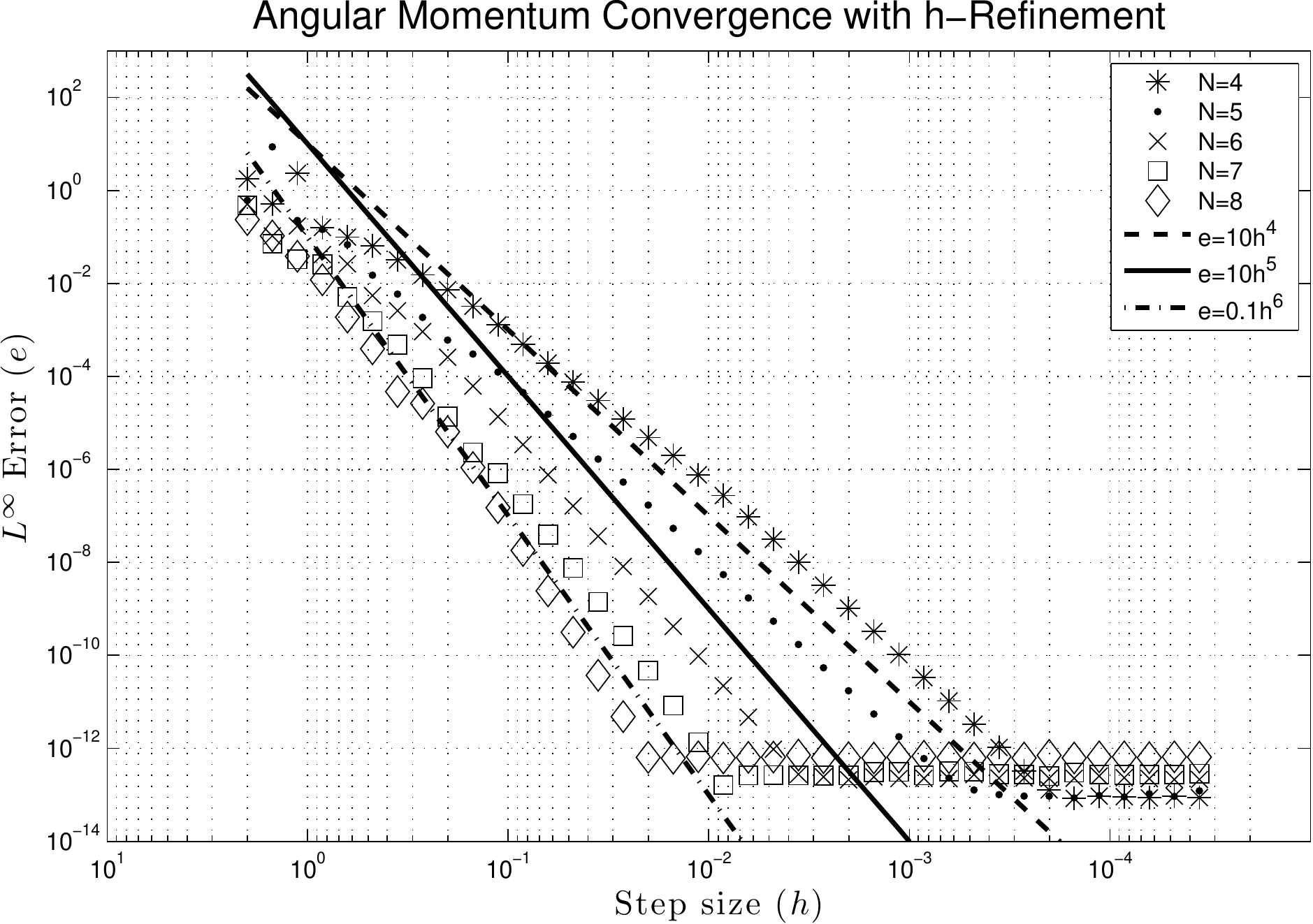}
  \caption{Convergence of the angular momentum Kepler 2-body problem with eccentricity 0.6 over 100 steps of \(h = 2.0\).}
  \label{fig:Kepler2bodyhRefinement_ang}
\end{figure}

\begin{figure}[htpb]
  \centering
  \includegraphics[width = 0.75\textwidth]{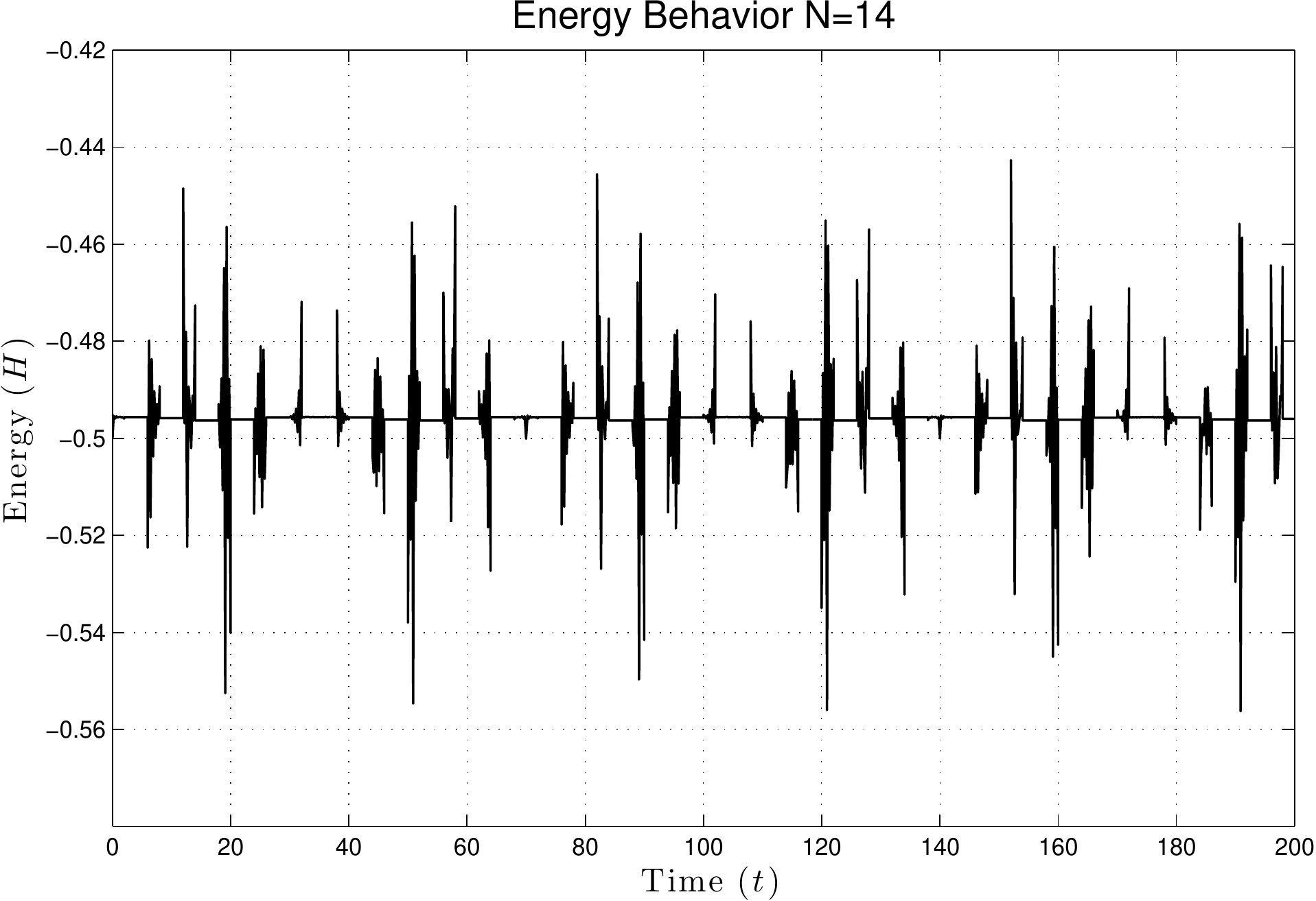}
  \caption{Stability of energy for Kepler 2-body problem.}
  \label{fig:Kepler2bodyEnergyStable}
\end{figure}


The first experiment we will examine is the choice of parameters \(D = 2\), \(m_{1} = m_{2} = 1\). Centering the coordinate system at \(q_{1}\),  we choose \(q_{2}\lefri{0} = \lefri{0.4,0}\), \(\dot{q}_{2}\lefri{0} = \lefri{0, 2}\), which has a known closed form solution which is a stable closed elliptical orbit with eccentricity \(0.6\). Knowing the closed form solution allows us to examine the rate of convergence to the true solution, and when solved with the large time step \(h = 2.0\), over 100 steps, the error of the one-step map is \(\mathcal{O}\lefri{0.56^{n}}\) with \(n\)-refinement and \(\mathcal{O}\lefri{h^{2\left\lceil\frac{n}{2}\right\rceil}}\) with \(h\)-refinement, as can be seen in Figure \ref{fig:Kepler2bodyNRefinement} and Figure \ref{fig:Kepler2bodyhRefinement}, respectively. The numerical evidence suggests that our bound for the one-step map with \(h\)-refinement is not sharp, as the convergence of the one-step map is always even. Interestingly, it is also possible to observe the different convergence rates of the one-step map and the Galerkin curves with \(n\)-refinement, as eventually the Galerkin curves have error approximately \(\mathcal{O}\lefri{0.74^{n}}\) while the one-step map has error approximately \(\mathcal{O}\lefri{0.56^{n}}\), and \(\sqrt{0.56} \approx 0.7483\). However, it appears that the error from the one step map dominates until very high choices of \(n\), and thus it is difficult to observe the error of the Galerkin curves directly with \(h\)-refinement, round off error becomes a problem before the error of the Galerkin curves does for smaller choices of \(n\).

The N-body Lagrangian is invariant under the action of \(\mbox{SO}\lefri{D}\), which yields the conserved Noether quantity of angular momentum. For the two body problem this is:
\begin{align*}
L\lefri{q,\dot{q}} =  q_{x}\dot{q}_{y} - q_{y}\dot{q}_{x}
\end{align*}%
where \(q = \lefri{q_{x},q_{y}}\). Numerical experiments show that the error of the angular momentum does not grow with the number of steps taken in the integration, Figure \ref{fig:Kepler2bodyLStable}, but that the error is of the same order as the error of Galerkin curve with \(n\)-refinement in Figure \ref{fig:Kepler2bodyNRefinement_Angular}. With \(h\)-refinement, the angular momentum appears to have error \(\mathcal{O}\lefri{h^{\frac{n}{2}} + 2}\) in Figure \ref{fig:Kepler2bodyhRefinement_ang}. This is interesting because the theoretical bound on the error of the Galerkin curves is \(\mathcal{O}\lefri{h^\frac{n}{2}}\), and the error of the Noether quantities is theoretically a factor \(C\lefri{h}\) times the error of the Galerkin curves, where \(C\) is the factor that arises in the proof of the convergence of the conserved Noether quantities. Numerical experiments suggest \(C\) is \(\mathcal{O}\lefri{h^{2}}\) for this problem, but that the Galerkin curves do converge at a rate of \(\mathcal{O}\lefri{h^{\frac{n}{2}}}\), which is consistent with of the Galerkin curve error estimate. While this evidence is not conclusive, it is suggestive that the error analysis provides a plausible bound. A careful analysis of the factor \(C\) would be an interesting direction for further investigation.

\subsubsection{The Solar System}

To illustrate the excellent stability proprieties of spectral variational integrators, we let \(D = 3\), \(N = 10\), and use the velocities, positions, and masses of the sun, 8 planet, and the dwarf planet Pluto on January 1, 2000 (as provided by the JPL Solar System ephemeris \cite{JPLEph}) as initial configuration parameters for the Kepler system. Taking \(100\) time steps of \(h = 100\) days, the \(N = 25\) spectral variational integrator produces a highly stable flow in Figure \ref{fig:InnerSolar}. It should be noted that orbits are closed, stable, and exhibit almost none of the ``precession'' effects that are characteristic of symplectic integrators, even though the time step is larger than the orbital period of Mercury. Additionally, considering just the outer solar system (Jupiter, Saturn, Uranus, Neptune, Pluto), and aggregating the inner solar system (Sun, Mercury, Venus, Earth, Mars) to a point mass, an \(N = 25\) spectral variational integrator taking 100 time steps \(h = 1825\) days (5 year steps) produces the orbital flow seen in Figure \ref{fig:OuterSolar}. Again, these are highly stable, precession free orbits. As can be clearly seen, the spectral variational integrator produces extremely stable flows, even for very large time steps.

\begin{figure}[t]
  \centering
  \includegraphics[width = 0.75\textwidth]{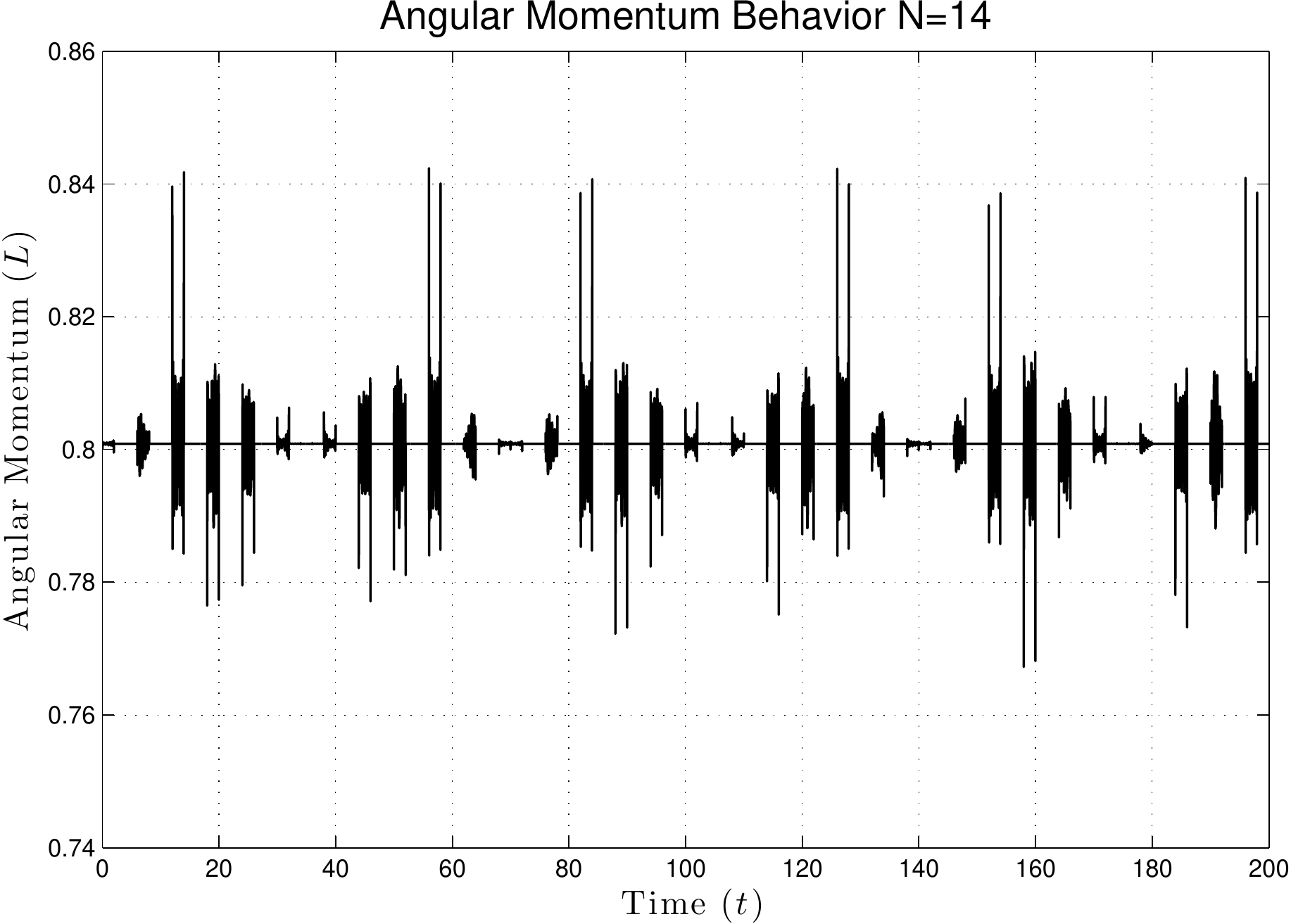}
  \caption{Stability of angular momentum for Kepler 2-body problem.}
  \label{fig:Kepler2bodyLStable}
\end{figure}

\afterpage{\clearpage}

\begin{figure}[p]
  \centering
  \includegraphics[width = 0.75\textwidth]{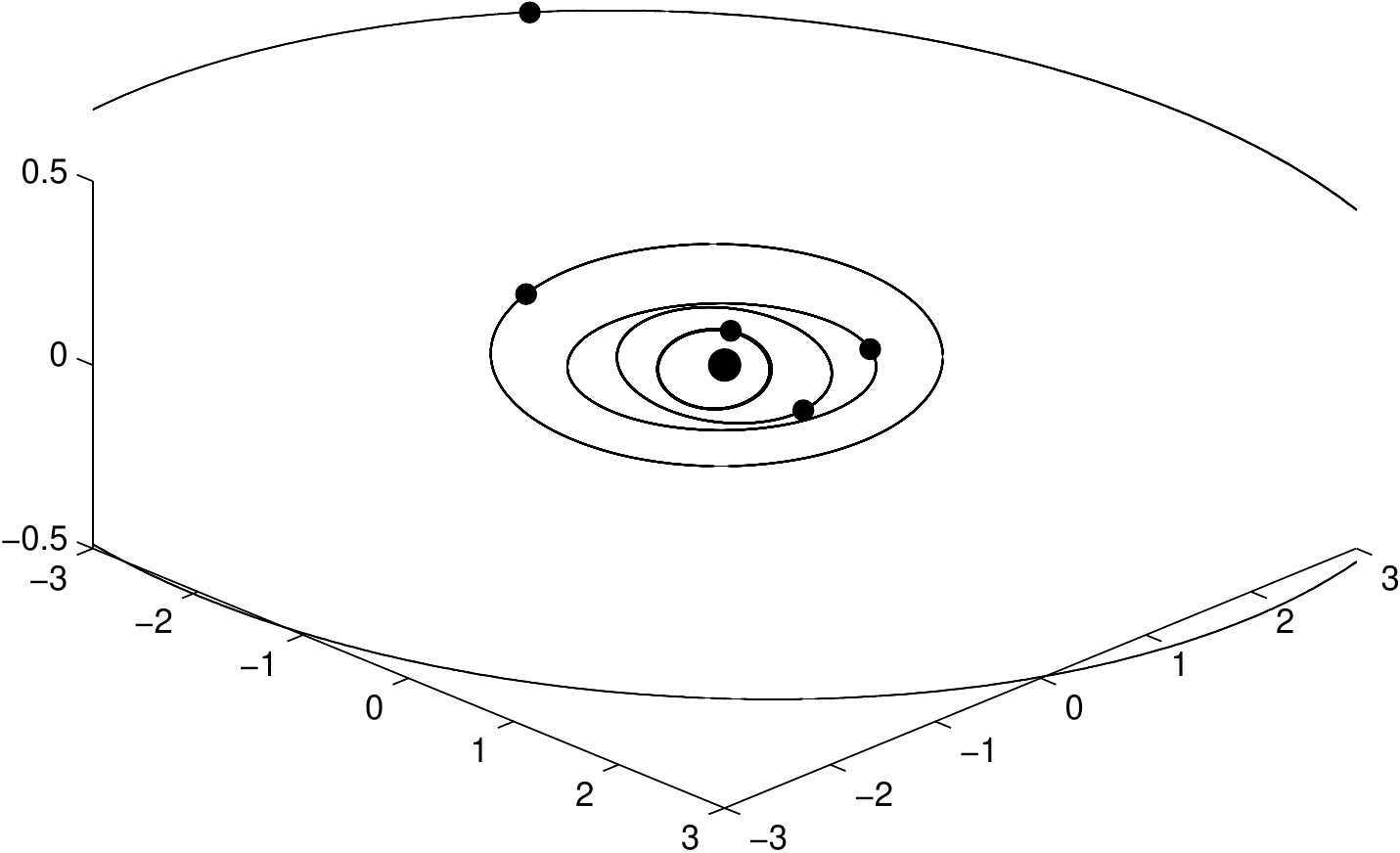}
  \caption{Orbital diagram for the inner solar produced by the spectral variational integrator using all 8 planets, the sun, and Pluto with 100 time steps with \(h=100\) days.}
  \label{fig:InnerSolar}
\end{figure}

\begin{figure}[p]
  \centering
  \includegraphics[width = 0.75\textwidth]{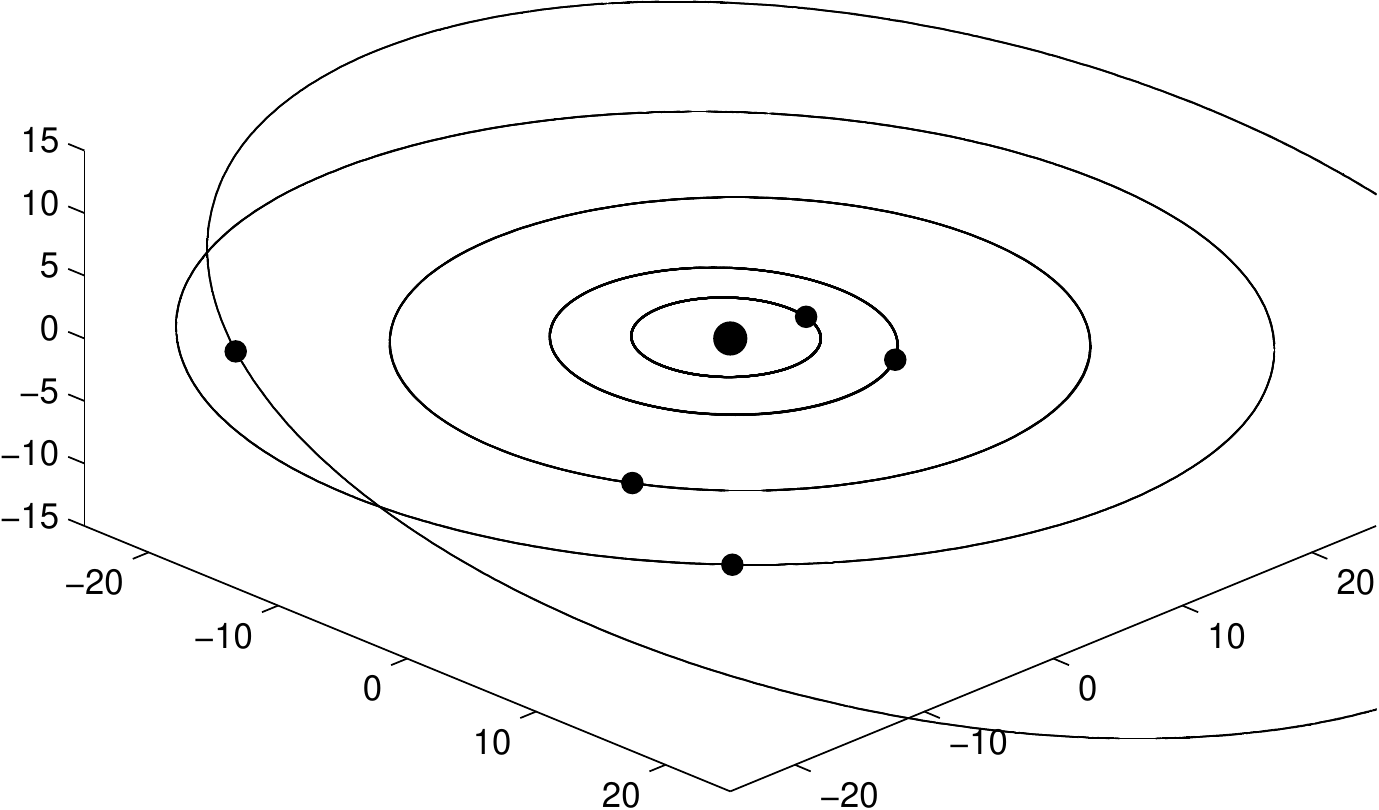}
  \caption{Orbital diagram for the inner solar produced by the spectral variational integrator using the 4 outer planets, Pluto, with the sun and 4 inner planets aggregated to a point with 100 time steps at \(h=1825\) days.}
  \label{fig:OuterSolar}
\end{figure}

\section{Conclusions and Future Work}

In this paper a new numerical method for variational problems was introduced, specifically a symplectic momentum-preserving integrator that exhibits geometric convergence to the true flow of a system under the appropriate conditions. These integrators were constructed under the general framework of Galerkin variational integrators, and made use of the global function paradigm common to many different spectral methods.

Additionally, a general convergence theorem was established for Galerkin type variational integrators, establishing the important result that, under suitable hypotheses, Galerkin variational integrators will inherit the optimal order of convergence permitted by the underlying approximation space used in their construction. This result provides a powerful tool for both constructing and analyzing variational integrators, it provides a methodology for constructing methods of very high order of accuracy, and it also establishes order of convergence for methods that can be viewed as Galerkin variational integrators. It was shown that from the one step map, a continuous approximation to the solution of the Euler-Lagrange equations can be easily recovered over each time step. The error of these continuous approximations was shown to be related to the error of the one step map. Furthermore, the Noether quantities along this continuous approximation approximate the true Noether quantity up to a small error which does not grow with the number of steps taken. It was also shown that the error of the Noether quantities converges to zero with \(n\) or \(h\) refinement at a predictable rate.

In addition to the convergence results, another interesting feature of spectral variational integrators is the construction of very high order methods that remain stable and accurate using time steps that are orders of magnitude larger than can be tolerated by traditional integrators. The trade off is that the computational effort required to compute each time step is also orders of magnitude larger than that of other methods, which are a major trade off in terms of the practicality of spectral variational integrators. However, a mitigating factor of this trade off is that the approach of solving a short sequence of large problems, as opposed to a large sequence of small problems, lends itself much better to parallelization and computational acceleration. The literature on methods for acceleration of the construction and solution of structured systems of linear and nonlinear problems for PDE problems is extensive, and it is likely that such methods could be applied to spectral variational integrators to greatly improve their computation cost.

\subsection{Future Work}

Future directions for this work are numerous. Because of generality of the construction of Galerkin variational integrators, there exists many possible directions of further exploration.

\subsubsection{Lie Group Spectral Variational Integrators} Following the approach of \citet{LeSh2011b} or \citet{BoMa2009}, it is relatively straight forward to extend spectral variational integrators to Lie groups using natural charts. A systematic investigation of the resulting Lie group methods, including convergence and near conservation of Noether quantities, would be a natural extension of the work done here.

\subsubsection{Novel Variational Integrators} The power of the Galerkin variational framework is its high flexibility in the choice of approximation spaces and quadrature rules used to construct numerical methods. This flexibility allows for the construction of novel methods specifically tailored to certain applications. One immediate example is the use of periodic functions to construct methods for detecting choreographies in Kepler problems, which would be closely related to methods already used with great success to detect choreographies. Another interesting application would be the use of high order polynomials to develop integrators for high-order Lie group problems, such as the construction of Riemannian splines, which has a variety of applications in motion planning. Enriching traditional polynomial approximation spaces with highly oscillatory functions could be used to develop methods for problems with dynamics evolving on radically different time scales, which are also very challenging for traditional numerical methods. 

\subsubsection{Multisymplectic Variational Integrators} Multisymplectic geometry (see \citet{MaPaSh1998}) has become an increasingly popular framework for extending much of the geometric theory from classical Lagrangian mechanics to Lagrangian PDEs. The foundations for a discrete theory have been laid, and there have been significant results achieved in geometric techniques for structured problems such as elasticity, fluid mechanics, non-linear wave equations, and computational electromagnetism. However, there is still significant work to be done in the areas of construction of numerical methods, analysis of discrete geometric structure, and especially error analysis. Galerkin type methods have become a standard method in classical numerical PDE methods, such as Finite-Element Methods, Spectral, and Pseudospectral methods. The variational Galerkin framework could provide a natural framework for extending these classical methods to structure preserving geometric methods for PDEs, and the analysis of such methods will rely on the notion of the boundary Lagrangian (see \citet{VaLiLe2011}), which is the PDE analogue of the exact discrete Lagrangian.

\section*{Acknowledgements}
This work was supported in part by NSF Grants CMMI-1029445, DMS-1065972, and NSF CAREER Award DMS-1010687.
\bibliographystyle{plainnat}
\bibliography{HaLe2012_SVI}

\appendix
\section{Proofs of Geometric Convergence of Spectral Variational Integrators}%
As stated in \S\ref{ConSection}, it can be shown that spectral variational integrators converge geometrically to the true flow associated with a Lagrangian under the appropriate conditions. The proof is similar to that of order optimality, and is offered below.

However, before we offer a proof of the theorem, we must establish a result that extends Theorem \ref{MarsConv}. Specifically, we must show:
\begin{theorem} \emph{(Extension of Theorem \ref{MarsConv} to Geometric Convergence)} \label{MarsExt}%
Given a regular Lagrangian \(L\) and corresponding Hamiltonian \(H\), the following are equivalent for a discrete Lagrangian \(L_{d}\lefri{q_{0},q_{1},n}\):%
\begin{enumerate}
\item there exist a positive constant \(K\), where \(K < 1\), such that the discrete Hamiltonian map for \(L_{d}\lefri{q_{0},q_{h},n}\) has error \(\mathcal{O}\lefri{K^{n}}\), 
\item there exists a positive constant \(K\), where \(K < 1\), such that the discrete Legendre transforms of \(L_{d}\lefri{q_{0},q_{h},n}\) have error \(\mathcal{O}\lefri{K^{n}}\),
\item there exists a positive constant \(K\), where \(K < 1\), such that \(L_{d}\lefri{q_{0},q_{h},n}\) approximates the exact discrete Lagrangian \(L_{d}^{E}\lefri{q_{0},q_{h},h}\) with error \(\mathcal{O}\lefri{K^{n}}\).
\end{enumerate}%
\end{theorem}%
The proof of this theorem is a simple modification of the proof of Theorem \ref{MarsConv}, and is included here for completeness. For details, the interested reader is referred to \cite{MaWe2001}.
\begin{proof}
Since we are assuming that the time step \(h\) is being held constant, we will suppress it as an argument to the exact discrete Lagrangian, writing \(L^{E}_{d}\lefri{q_{0},q_{h}}\) for \(L^{E}_{d}\lefri{q_{0},q_{h},h}\). First, we will assume that \(L_{d}\lefri{q_{0},q_{h},n}\) approximates \(L_{d}\lefri{q_{0},q_{h}}\) with error \(\mathcal{O}\lefri{K^{n}}\) and show this implies the discrete Legendre transforms have error \(\mathcal{O}\lefri{K^{n}}\). By assumption, if \(L_{d}\lefri{q_{0},q_{h},n}\) has error \(\mathcal{O}\lefri{K^{n}}\), there exists a function which is smooth in its first two arguments \(e_{v}: Q \times Q \times \mathbb{N} \rightarrow \mathbb{R}\) such that:%
\begin{align*}
L_{d}\lefri{q_{0},q_{h},n} = L_{d}^{E}\lefri{q_{0},q_{h}} + K^{n}e_{v}\lefri{q_{0},q_{h},n},
\end{align*}%
with \(\left|e_{v}\lefri{q_{0},q_{h},n}\right| \leq C_{v}\) on \(U_{v}\). Taking derivatives with respect to the first argument yields:%
\begin{align*}
\mathbb{F}^{-}L^{n}_{d}\lefri{q_{0},q_{h}} = \mathbb{F}^{-}L_{d}^{E}\lefri{q_{0},q_{h}} + K^{n}D_{1}e_{v}\lefri{q_{0},q_{h},n},
\end{align*}%
and with respect to the second yields:%
\begin{align*}
\mathbb{F}^{+}L^{n}_{d}\lefri{q_{0},q_{h}} = \mathbb{F}^{+}L_{d}^{E}\lefri{q_{0},q_{h}} + K^{n}D_{2}e_{v}\lefri{q_{0},q_{h},n}.
\end{align*}%
Since \(e_{v}\) is smooth and bounded over the closed set \(U\), so are \(D_{1}e_{v}\) and \(D_{2}e_{v}\), yielding that the discrete Legendre transforms have error \(\mathcal{O}\lefri{K^{n}}\). Now, to show that if the discrete Legendre transforms have error \(\mathcal{O}\lefri{K^{n}}\), the discrete Lagrangian has error \(\mathcal{O}\lefri{K^{n}}\), we write:
\begin{align*}
e_{v}\lefri{q_{0},q_{h},n} &= \frac{1}{K^{n}}\left[L_{d}\lefri{q_{0},q_{h},n} - L^{E}_{d}\lefri{q_{0},q_{h}}\right],\\
D_{1}e_{v}\lefri{q_{0},q_{h},n} &= \frac{1}{K^{n}}\left[\mathbb{F}^{-}L_{d}\lefri{q_{0},q_{h},n} - \mathbb{F}^{-}L^{E}_{d}\lefri{q_{0},q_{h}}\right], \\
D_{2}e_{v}\lefri{q_{0},q_{h},n} &= \frac{1}{K^{n}}\left[\mathbb{F}^{+}L_{d}\lefri{q_{0},q_{h},n} - \mathbb{F}^{+}L^{E}_{d}\lefri{q_{0},q_{h}}\right].
\end{align*}
Since \(D_{1}e_{v}\) and \(D_{2}e_{v}\) are smooth and bounded on a bounded set, this implies there exists a function \(d\lefri{n}\) such that
\begin{align*}
\left\|e_{v}\lefri{q\lefri{0},q\lefri{h},n} - d\lefri{n}\right\| \leq C_{v},
\end{align*}
for some constant \(C_{v}\). This shows that the discrete Lagrangian is equivalent to a discrete Lagrangian with error \(\mathcal{O}\lefri{K^{n}}\). We note that the equivalence is a consequence of the fact that one can add a function of \(h\) or \(n\) to any discrete Lagrangian and the resulting discrete Euler-Lagrange equations and discrete Legendre Transforms are unchanged, hence the function \(d\lefri{n}\).

To show the equivalence of the discrete Hamiltonian map having error \(\mathcal{O}\lefri{K^{n}}\) and the discrete Legendre transforms having error \(\mathcal{O}\lefri{K^{n}}\), we recall expressions for the discrete Hamiltonian map for the discrete Lagrangian \(L_{d}\) and exact discrete Lagrangian \(L^{E}_{d}\):
\begin{align*}
F_{L_{d}} &= \mathbb{F}^{+}L_{d} \circ \lefri{\mathbb{F}^{-}L_{d}}^{-1}, \\
F_{L^{E}_{d}} &= \mathbb{F}^{+}L^{E}_{d} \circ \lefri{\mathbb{F}^{-}L^{E}_{d}}^{-1}.
\end{align*}
Now, we make use of the following consequence of the implicit function theorem: If we have smooth functions \(g_{1},g_{2}\) and the sequences of functions \(\left\{f_{1_{n}}\right\}_{n=1}^{\infty}\), \(\left\{f_{2_{n}}\right\}_{n=1}^{\infty}\), \(\left\{e_{1_{n}}\right\}_{n=1}^{\infty}\) and \(\left\{e_{2_{n}}\right\}_{n=1}^{\infty}\) such that
\begin{align*}
f_{1_{n}}\lefri{x} &= g_{1}\lefri{x} + K^{n}e_{1_{n}}\lefri{x}, \\
f_{2_{n}}\lefri{x} &= g_{2}\lefri{x} + K^{n}e_{2_{n}}\lefri{x},
\end{align*}
where \(\sup{\left\{\left\|e_{1_{n}}\right\|\right\}_{n=1}^{\infty}} < C_{1}\) and \(\sup{\left\{\left\|e_{2_{n}}\right\|\right\}_{n=1}^{\infty}} < C_{2}\) on compact sets, then
\begin{align}
f_{2_{n}}\lefri{f_{1_{n}}\lefri{x}} &= g_{2}\lefri{g_{1}\lefri{x}} + K^{n}e_{12_{n}}\lefri{x} \label{ImpFun1}\\
f_{1_{n}}^{-1}\lefri{y} &= g_{1}^{-1}\lefri{y} + K^{n}\bar{e}_{1_{n}}\lefri{y} \label{ImpFun2}
\end{align}
for some sequences of functions \(\left\{e_{12_{n}}\right\}_{n=1}^{\infty}\), \(\left\{\bar{e}_{1_{n}}\right\}_{n=1}^{\infty}\) where \(\sup{\left\{\left\|e_{12_{n}}\right\|\right\}_{n=1}^{\infty}} < C_{1}\) and \(\sup{\left\{\left\|\bar{e}_{1_{n}}\right\|\right\}_{n=1}^{\infty}} < C_{2}\) on compact sets.
 
It follows from (\ref{ImpFun1}) and (\ref{ImpFun2}) that if the discrete Legendre transforms have error \(\mathcal{O}\lefri{K^{n}}\), the discrete Hamiltonian map does as well. Finally, if we have a discrete Hamiltonian map has error \(\mathcal{O}\lefri{K^{n}}\), we use the identity
\begin{align*}
\lefri{\mathbb{F}^{-}L_{d}}^{-1}\lefri{q_{0},p_{0}} &= \lefri{q_{0}, \pi_{Q}\circ F_{L_{d}}\lefri{q_{0},p_{0}}}
\end{align*}
where \(\pi_{Q}\) is the projection map \(\pi_{Q}:\lefri{q,p} \rightarrow q\) and (\ref{ImpFun2}) to see that \(\mathbb{F}^{-}L_{d}\) is has error \(\mathcal{O}\lefri{K^{n}}\), and the identity:
\begin{align*}
\mathbb{F}^{+}L_{d} = F_{L_{d}} \circ \mathbb{F}^{-}L_{d},
\end{align*}
along with (\ref{ImpFun1}) to establish that \(F^{+}L_{d}\) also has error \(\mathcal{O}\lefri{K^{n}}\), which completes the proof.
\end{proof}%
This simple extension is a critical tool for establishing the geometric convergence of spectral variational integrators, and leads to the following theorem concerning the accuracy of spectral variational integrators.%
%
%
\begin{theorem} \emph{(Geometric Convergence of Spectral Variational Integrators)} \label{SpecConv_App}
Given an interval \(\left[0,h\right]\) and a Lagrangian \(L:TQ \rightarrow \mathbb{R}\), let \(\truq\) be the exact solution to the Euler-Lagrange equations, and \(\galqn\) be the stationary point of the spectral variational discrete action
\begin{align*}
\SDLN{q_{0}}{q_{h}}{n} = \ext_{\galargsf{q_{0}}{q_{h}}} \mathbb{S}_{d}\lefri{\left\{q_{i}\right\}_{i=1}^{n}}= \ext_{\galargsf{q_{0}}{q_{h}}} h\sum_{j=0}^{m_{n}} \bnj L\lefri{q_{n}\lefri{\cnjh},\dot{q}_{n}\lefri{\cnjh}}.
\end{align*}
If:%
\begin{enumerate}
\item there exists constants \(\ApproxC,\ApproxK\), \(\ApproxK < 1\), independent of \(n\), such that, for each \(n\), there exists a curve \(\optqn \in \mathbb{M}^{n}\lefri{\left[0,h\right],Q}\) such that,%
\begin{align*}
  \left\|\lefri{\truq,\dtruq} - \lefri{\optqn,\doptqn}\right\| & \leq  \ApproxC \ApproxK^{n},
\end{align*}
\item there exists a closed and bounded neighborhood \(U \subset TQ\) such that \(\lefri{\truq\lefri{t},\dtruq\lefri{t}} \in U\) and \(\lefri{\optqn\lefri{t},\doptqn\lefri{t}} \in U\) for all \(t\) and \(n\), and all partial derivatives of \(L\) are continuous on \(U\),
\item for the sequence of quadrature rules \(\mathcal{G}_{n}\lefri{f} = h\sum_{j=1}^{m_{n}} \bnj f\lefri{\cnjh} \approx \int_{0}^{h}f\lefri{t}\dt\)  there exists constants \(\QuadC\), \(\QuadK\), \(\QuadK < 1\), independent of \(n\) such that:%
\begin{align*}
  \left|\int_{0}^{h} L\lefri{q_{n}\lefri{t}, \dot{q}_{n}\lefri{t}}\dt - h\sum_{j=1}^{m_{n}} \bnj L\lefri{q_{n}\lefri{\cnjh},\dot{q}_{n}\lefri{\cnjh}}\right| \leq \QuadC \QuadK^{n},
\end{align*}%
for any \(q_{n}\in \mathbb{M}^{n}\lefri{\left[0,h\right],Q}\),
\item and the stationary points \(\truq\), \(\galqn\) minimize their respective actions,%
\end{enumerate}%
then%
\begin{align}
  \left|\EDL{q_{0}}{q_{1}} - \SDLN{q_{0}}{q_{1}}{n}\right| \leq \SpecC\SpecK^{n} \label{GeoBound_App}
\end{align}%
for some constants \(\SpecC,\SpecK\), \(\SpecK < 1\), independent of \(n\), and hence the discrete Hamiltonian flow map has error \(\mathcal{O}\lefri{\SpecK^{n}}\).%
\end{theorem}
%
%
\begin{proof}%
As before, we rewrite both the exact discrete Lagrangian and the spectral discrete Lagrangian:%
\begin{align*}
\left|\EDL{q_{0}}{q_{1}} - \SDLN{q_{0}}{q_{1}}{n}\right| &= \left|\int_{0}^{h}L\lefri{\truq\lefri{t},\dtruq\lefri{t}}\dt - \mathcal{G}_{n}\lefri{L\lefri{\galqn\lefri{t},\dgalqn\lefri{t}}}\right| \\
&= \left|\int_{0}^{h}L\lefri{\truq\lefri{t},\dtruq\lefri{t}}\dt - h\sum_{j=1}^{m_{n}}\bnj L\lefri{\galqn\lefri{\cnjh},\dgalqn\lefri{\cnjh}}\right| \\
&= \left|\int_{0}^{h}L\lefri{\truq,\dtruq}\dt - h\sum_{j=1}^{m_{n}}\bnj L\lefri{\galqn,\dgalqn}\right|,
\end{align*}%
with suppression of the \(t\) argument. We introduce the action evaluated on the curve \(\optqn\):%
\begin{align}
 \left|\int_{0}^{h}L\lefri{\truq,\dtruq}\dt - h\sum_{j=1}^{m_{n}}\bnj L\lefri{\galqn,\dgalqn}\right| & = \left|\int_{0}^{h}L\lefri{\truq,\dtruq}\dt - \int_{0}^{h} L\lefri{\optqn, \doptqn} \dt \right. \\
& \hspace{10em} \left. + \int_{0}^{h} L\lefri{\optqn,\doptqn}\dt - h\sum_{j=1}^{m_{n}}\bnj L\lefri{\galqn,\dgalqn}\right|  \nonumber \\
&\leq \left|\int_{0}^{h}L\lefri{\truq,\dtruq}\dt - \int_{0}^{h} L \lefri{\optqn, \doptqn}\dt\right| \\
& \hspace{10em} + \left|\int_{0}^{h} L \lefri{\optqn,\doptqn} \dt - h\sum_{j=1}^{m}\bnj L\lefri{\galqn,\dgalqn}\right|. \label{STwoTerms}
\end{align}
Considering the first term in (\ref{STwoTerms}):%
\begin{align*}
  \left|\int_{0}^{h}L\lefri{\truq,\dtruq}\dt - \int_{0}^{h} L \lefri{\optqn, \doptqn}\dt\right| &= \left|\int_{0}^{h}L\lefri{\truq,\dtruq} - L \lefri{\optqn, \doptqn} \dt\right| \\
& \leq  \int_{0}^{h} \left|L\lefri{\truq,\dtruq} - L\lefri{\optqn,\doptqn}\right|\dt. 
\end{align*}%
By assumption, all partials of \(L\) are continuous on \(U\), and since \(U\) is closed and bounded, this implies \(L\) is Lipschitz on \(U\), so let \(\LagLipC\) denote the Lipschitz constant. Since, again by assumption, \(\lefri{\truq,\dtruq} \in U\) and \(\lefri{\optqn,\doptqn} \in U\), we can obtain:%
\begin{align*}
\int_{0}^{h}\left|L\lefri{\truq,\dtruq} - L \lefri{\optqn,\doptqn}\right|\dt \leq&  \int_{0}^{h} \LagLipC \left|\lefri{\truq,\dtruq} - \lefri{\optqn,\doptqn}\right| \dt \\
\leq& \int_{0}^{h} \LagLipC \ApproxC \ApproxK^{n} \dt \\
=& h\LagLipC\ApproxC\ApproxK^{n}.
\end{align*} 
Hence,
\begin{align}
\left|\int_{0}^{h} L\lefri{\truq, \dtruq} \dt - \int_{0}^{h} L\lefri{\optqn,\doptqn} \dt\right| \leq h \LagLipC \ApproxC \ApproxK^{n}. \label{SFirstTermIneq}
\end{align}%
Next, considering the second term in (\ref{STwoTerms}),%
\begin{align*}
\left|\int_{0}^{h} L\lefri{\optqn,\doptqn} \dt - \sum_{j=1}^{m} h\bnj L\lefri{\galqn,\dgalqn} \right|,
\end{align*}%
since \(\galqn\) minimizes its action,%
\begin{align}
h\sum_{j=1}^{m_{n}} \bnj L\lefri{\galqn,\dgalqn} \leq h\sum_{j=1}^{m_{n}} \bnj L \lefri{\optqn,\doptqn} \leq \int_{0}^{h} L\lefri{\optqn,\doptqn} \dt + \QuadC \QuadK^{n} \label{SUpperBound}
\end{align}%
where the inequalities follow from the assumptions on the order of the quadrature rule and (\ref{SFirstTermIneq}). Furthermore,%
\begin{align}
h\sum_{j=1}^{m_{n}} \bnj L \lefri{\galqn,\dgalqn} \geq \int_{0}^{h} L\lefri{\galqn,\dgalqn} \dt - \QuadC\QuadK^{n} &\geq \int_{0}^{h} L\lefri{\truq,\dtruq}  \dt - \QuadC \QuadK^{n} \nonumber\\
&\geq \int_{0}^{h} L\lefri{\optqn,\doptqn} \dt - h\LagLipC \ApproxC \ApproxK^{n} - \QuadC \QuadK^{n}, \label{SLowerBound}
\end{align}%
where the inequalities follow from (\ref{SFirstTermIneq}), the order of the sequence of quadrature rules, and the assumption that \(\truq\) minimizes its action. Putting (\ref{SUpperBound}) and (\ref{SLowerBound}) together, we can conclude:%
\begin{align}
\left|\int_{0}^{h}L\lefri{\optqn,\doptqn}\dt - h\sum_{j=1}^{m_{n}} \bnj L\lefri{\galqn,\dgalqn}\right| \leq \lefri{h\LagLipC\ApproxC + \QuadC}\SpecK^{-n} \label{SSecondTermIneq}.
\end{align}%
where \(\SpecK = \max\lefri{\ApproxK,\QuadK}\). Now, combining the bounds (\ref{SFirstTermIneq}) and (\ref{SSecondTermIneq}) in (\ref{STwoTerms}), we can conclude
\begin{align*}
\left|\EDL{q_{0}}{q_{1}} - \SDLN{q_{0}}{q_{1}}{n}\right| \leq \lefri{2h\LagLipC\ApproxC + \QuadC}\SpecK^{-n}
\end{align*}
which, combined with Theorem \ref{MarsExt}, establishes the rate of convergence.
\end{proof}%
%
%

\end{document}